\documentclass[12pt,reqno]{amsart}
\allowdisplaybreaks[2]

\makeatletter
\@namedef{subjclassname@2020}{%
	\textup{2020} Mathematics Subject Classification}
\makeatother
\usepackage{caption}
\usepackage{amssymb,bm,mathrsfs}
\usepackage{float}
\usepackage{enumerate}
\usepackage{amsfonts}
\usepackage{color}
\usepackage{tikz}
\usepackage{pgfplots}
\pgfplotsset{compat=1.3}
\definecolor{spectrumblue}{RGB}{25,83,133}
\definecolor{spectrumorange}{RGB}{174,78,24}
\definecolor{spectrumgreen}{RGB}{36,112,85}
\usepackage[centering,width=6in]{geometry}
\usepackage[colorlinks,pdfstartview=FitH]{hyperref}

\hypersetup{
	colorlinks,
	citecolor=blue,
	filecolor=black,
	linkcolor=blue,
	urlcolor=black
}

\newtheorem{theorem}{Theorem}[section]
\newtheorem{lemma}[theorem]{Lemma}
\newtheorem{proposition}[theorem]{Proposition}
\newtheorem{corollary}[theorem]{Corollary}
\newtheorem{definition}{Definition}[section]

\theoremstyle{remark}
\newtheorem{remark}{Remark}[section]
\newtheorem{example}{Example}[section]
\numberwithin{equation}{section}

\newcommand{\Pck}{\mathcal P}
\newcommand{\card}{\#}
\DeclareMathOperator{\usdim}{\overline{\mathrm{Dim}}_{\theta}}

\DeclareMathOperator{\u0dim}{\overline{\mathrm{Dim}}_{0}}
\DeclareMathOperator{\uodim}{\overline{\mathrm{Dim}}_{1}}

\DeclareMathOperator{\l0dim}{\underline{\mathrm{Dim}}_{0}}

\DeclareMathOperator\hdim{\mathrm{Dim}_{\mathrm{H}}}
\newcommand{\dist}{\operatorname{dist}}
\newcommand{\diam}{\operatorname{diam}}
\newcommand{\uDim}{\overline{\operatorname{Dim}}}
\newcommand{\lDim}{\underline{\operatorname{Dim}}}

\def\N{\mathbb{N}}
\def\Z{\mathbb{Z}}
\def\R{\mathbb{R}}

\begin{document}

	\title[ Discrete intermediate dimensions]{Discrete intermediate dimensions: endpoints, spectra and quasi-isometric invariance}

	\author{Jinjun Li}
	
	\address[Jinjun Li]{School of Mathematics and Information Science, Guangzhou University, Guangzhou, 510006, P.~R.~China}
	\email{lijinjun@gzhu.edu.cn}
	
	\author[Lifeng Xi]{Lifeng Xi$^\ast$}
	\address[Lifeng Xi]{School of Mathematics and Statistics, Ningbo University,\newline Ningbo, 315211, P.~R.~China}
	\email{xilifeng@nbu.edu.cn}
	
	\author{Shuqin Zhang}
	
	\address[Shuqin Zhang]{School of Mathematics and Information Science, Guangzhou University, Guangzhou, 510006, P.~R.~China}
	\email{zhangshuqin@fudan.edu.cn}

	\thanks{$^\ast$Corresponding author. Email: xilifengningbo@yahoo.com; xilifeng@nbu.edu.cn}
	
%	\thanks{Funding: National Natural Science Foundation of China
%(Grant Nos. 12171107 and 12571096).}

	

	\begin{abstract}
We introduce upper and lower discrete intermediate dimensions for unbounded
subsets of Euclidean space and extend them to metric spaces. The definitions
restrict the covering radii within a scale-dependent range of dyadic shells
and take a supremum over an auxiliary parameter. The resulting dimensions
are monotone in the parameter and recover the Barlow--Taylor discrete
Hausdorff dimension at zero. We prove continuity at every positive parameter,
with a sharp quantitative bound, and give an example with a jump at zero.
At the upper endpoint, the upper dimension equals upper mass dimension for
uniformly discrete Euclidean sets; the analogous lower mass identity fails.
We establish a mass distribution principle and prove quasi-isometric
invariance under a uniform coarse packing condition. For spaces with uniform
polynomial volume growth, we construct clustered subsets attaining the
continuity bound. Independent thinning truncates these spectra at the
resulting mass dimension, almost surely for all parameters simultaneously.
Applications include the discrete Heisenberg group with its word metric
and the ranges of simple symmetric random walks.
\end{abstract}
	\maketitle
	\section{Introduction and main results}

Hausdorff and box dimensions differ because they impose different restrictions
on covers. Falconer, Fraser and Kempton \cite{FFK} introduced intermediate
dimensions by restricting the relative sizes of the covering sets. Their
parameter $\theta\in[0,1]$ joins Hausdorff dimension to box dimension.
Applications include Lipschitz classification \cite{BK24}, projections
\cite{BFF}, self-affine codings \cite{Fengz24}, and stochastic processes
\cite{DK,Fal22}; see \cite{Fal21} for a survey. Banaji \cite{Bana23}
extended the covering restrictions by allowing more general functions to
relate the smallest and largest admissible scales. Banaji and Fraser
\cite{BF23} studied intermediate dimensions of infinitely generated
attractors, obtaining formulae for self-conformal limit sets under a
separation condition in terms of their Hausdorff dimension and the
intermediate dimensions of the set of fixed points.

For unbounded sets, the geometry at large scales replaces the geometry at
small scales. Barlow and Taylor \cite{BT89,BT} introduced discrete Hausdorff
and packing dimensions to study this geometry, notably for random-walk
ranges. Further applications appear in
\cite{BP17,DS,Daw,Kho94,KKX,LW1,LW,XZ}. Jiang, Miao and Xi \cite{JMX}
developed discrete dimensions and quasi-isometric invariance in metric
spaces.

Here we introduce upper and lower discrete intermediate dimensions.
The construction restricts both the covering radii and the shells being
examined, then takes a supremum over an auxiliary parameter. This last
step makes the dimensions monotone and permits the small parameter
changes needed to transport covers under quasi-isometries. We retain
this regularization throughout. The dimensions recover discrete Hausdorff
dimension at zero, are continuous on $(0,1]$, and may jump at zero.
Our main results also give sharp deterministic spectra and the almost-sure
spectra obtained by independently thinning clustered sets.

\subsection{Definitions}
We use $\N=\{1,2,\ldots\}$. For $n\in\N$, define the $n$-th shell of $\R^d$ by
	\[
	\text{$S_n=B(0,2^n)\setminus B(0,2^{n-1})$ and $S_0=B(0,1)$.}
	\]
	Throughout,  $\R^d$ is equipped with its usual Euclidean metric, and all balls are open.
	For \(0<\delta<1\), let \(N_\delta\in\N\) be the unique integer satisfying
	\begin{equation}\label{eq:Ndelta}
		2^{-N_\delta}\leq \delta<2^{-N_\delta+1}.
	\end{equation}

	For $\eta\in (0,1]$, define
	\[
	J_{\delta,\eta}=\left\{N_\delta, N_\delta+1,\ldots, \left\lceil \frac{N_\delta}{\eta}\right\rceil\right\},
	\]
	where $\lceil x \rceil$ is the least integer not smaller than $x$. For $\eta=0,$ define $J_{\delta,\eta}=\{n\in \N: n\ge N_\delta\}$.

	\begin{definition}\label{def:Phi}
		Let $E\subset\R^d$, $s\geq0$, and $\eta\in[0,1]$.  For $n\in J_{\delta,\eta}$, set
		\[
\begin{split}
\Phi_{n,\delta}^{s,\eta}(E)=\inf\biggl\{
&\sum_{i\in\mathcal I_n}(r_{n,i}/2^n)^s:
E\cap S_n\subseteq\bigcup_{i\in\mathcal I_n}B(x_i,r_{n,i}),\\
&x_i\in S_n,\quad 1\leq r_{n,i}\leq2^n\delta^\eta,
\quad\mathcal I_n\text{ finite}\biggr\}.
\end{split}
\]
		If $E\cap S_n=\emptyset$, set $\Phi_{n,\delta}^{s,\eta}(E)=0$.  Define the fixed-parameter shell content
		\begin{equation}\label{eq:Gamma-delta}
			\Gamma_{\delta}^{s,\eta}(E)
			:=\sum_{n\in J_{\delta,\eta}}\Phi_{n,\delta}^{s,\eta}(E),
		\end{equation}
		and define the regularized content by
		\begin{equation}\label{eq:Phi-delta}
			\Phi_{\delta}^{s,\theta}(E)
			:=\sup_{0\leq\eta\leq\theta}\Gamma_{\delta}^{s,\eta}(E),
			\qquad 0\leq\theta\leq1.
		\end{equation}
		The value $+\infty$ is allowed.
	\end{definition}
	\begin{remark}\label{fad}
		The shell content is finitely subadditive: $\Phi_{n,\delta}^{s,\eta}(E\cup F)\le \Phi_{n,\delta}^{s,\eta}(E)+\Phi_{n,\delta}^{s,\eta}(F)$.
	\end{remark}
	The regularization \eqref{eq:Phi-delta} ensures monotonicity of the resulting dimensions.
	
	Bounded Euclidean shells admit finite covers at every admissible radius.
	We use finite covers throughout; allowing countable covers gives the same
	infimum, since each ball has cost at least $2^{-ns}$ in a fixed shell.  The fixed-parameter series in \eqref{eq:Gamma-delta} may diverge when $\eta=0$, and therefore the regularized content in \eqref{eq:Phi-delta} is allowed to be infinite.
	
	\begin{definition}\label{def:dimensions}
		The upper and lower discrete intermediate dimensions of $E\subset \R^d$ are, respectively, defined by
		\[
		\uDim_{\theta}E=\inf\left\{s\geq0:\limsup_{\delta\searrow0}
		\frac{\log\Phi_{\delta}^{s,\theta}(E)}{-\log\delta}\leq 0\right\},
		\]
		and
		\[
		\lDim_{\theta}E =\inf\left\{s\geq0:\liminf_{\delta\searrow0}
		\frac{\log\Phi_{\delta}^{s,\theta}(E)}{-\log\delta}\leq 0\right\}.
		\]
	\end{definition}
	As usual, we adhere to the convention that $\log0=-\infty$ and $\log(+\infty)=+\infty$. An infimum over an empty set of exponents is $+\infty$.
	
	Throughout the paper, $\#A$ denotes the cardinality of $A$.
	We write $f\lesssim g$ if there exists a constant $C>0$ such that
	$f\le Cg$, and define $f\gtrsim g$ analogously.
	We write $f\asymp g$ if both $f\lesssim g$ and $f\gtrsim g$ hold. The notation $f\lesssim_\alpha g$ means that the implicit constant $C$ in the inequality depends on the parameter $\alpha.$
	\begin{remark}\label{rem:logs}
Bounded sets have dimension zero. The inequality $\leq0$ in
Definition~\ref{def:dimensions} is essential: if $E$ has one point in
each shell, then $\Phi_\delta^{s,\theta}(E)\asymp\delta^s$ for $s>0$,
whereas $\Phi_\delta^{0,\theta}(E)=+\infty$. Requiring equality to zero
would therefore give the wrong critical exponent.
\end{remark}

\subsection{Quasi-isometric invariance}
Let $X, Y$ be two metric spaces. A map $f:X\to Y$ is an $(L,c,C)$-\textit{quasi-isometry} if \(L\geq1\),
	\(c,C\geq0\), and
	\begin{equation}\label{eq:qi-distortion}
		\begin{aligned}
			&L^{-1}d_X(x,x')-c
			\leq d_Y(f(x),f(x'))
			\leq Ld_X(x,x')+c,\qquad(x,x'\in X).
		\end{aligned}
	\end{equation}
	while, for every \(y\in Y\), there exists \(x\in X\) such that
	\begin{equation}\label{eq:qi-density}
		d_Y(y,f(x))\leq C.
	\end{equation}
	This is equivalent, after increasing \(C\) by an arbitrarily small fixed
	amount, to the formulation
	\(\sup_{y\in Y}\dist(y,f(X))<\infty\) used in \cite{JMX}.
	
	Our first result concerns subsets with their induced Euclidean metrics.
\begin{theorem}\label{thm:main}
		Let $E\subset\R^d$ and $F\subset\R^m$ be unbounded sets equipped with the induced Euclidean metrics.  If $E$ and $F$ are quasi-isometric, then, for every $\theta\in[0,1]$,
		\[
		\uDim_{\theta}E=\uDim_{\theta}F
		\qquad\text{and}\qquad
		\lDim_{\theta}E=\lDim_{\theta}F.
		\]
	\end{theorem}
For the metric-space version, the definitions and the packing hypothesis
are as follows. An \emph{unbounded pointed metric space} $(X,d,o)$ is an
unbounded metric space together with a distinguished point $o\in X$.
No base-point preservation is required of a quasi-isometry.

\begin{definition}\label{def:content}\label{def:metric-dimensions}
Put $S_n^X(o)=B_X(o,2^n)\setminus B_X(o,2^{n-1})$ for $n\geq1$ and
$S_0^X(o)=B_X(o,1)$. Use $N_\delta$ and $J_{\delta,\eta}$ as above.
For $n\in J_{\delta,\eta}$, define
\begin{equation*}
\begin{split}
\Phi_{n,\delta,o}^{s,\eta}(X)=\inf\biggl\{
 &\sum_{i\in I}(r_i/2^n)^s:
 S_n^X(o)\subseteq\bigcup_{i\in I}B_X(x_i,r_i),\\
 &x_i\in S_n^X(o),\quad 1\leq r_i\leq2^n\delta^\eta,
 \quad I\text{ finite}\biggr\}.
\end{split}
\end{equation*}
An empty shell has content zero; an infimum over no finite admissible
covers is $+\infty$. Set
\begin{equation*}
\Gamma_{\delta,o}^{s,\eta}(X)=\sum_{n\in J_{\delta,\eta}}
\Phi_{n,\delta,o}^{s,\eta}(X),\qquad
\Phi_{\delta,o}^{s,\theta}(X)=\sup_{0\leq\eta\leq\theta}
\Gamma_{\delta,o}^{s,\eta}(X).
\end{equation*}
Define $\uDim_{\theta,o}X$ and $\lDim_{\theta,o}X$ by
Definition~\ref{def:dimensions}, with this content and the conventions
stated there. For $G\subset S_n^X(o)$, write
$\Phi_{n,\delta,o}^{s,\eta}(G)$ when only $G$ is covered;
the centers still lie in $S_n^X(o)$.
\end{definition}

\begin{definition}[Uniform coarse packing]\label{def:UCP}
Let $\Pck_r(A)$ be the supremum of the cardinalities of $r$-separated
subsets of $A$. A metric space $X$ satisfies the \emph{uniform coarse
packing condition} if, for every $\Lambda\geq1$,
\begin{equation}\label{eq:UCP}
M_X(\Lambda):=\sup_{x\in X,\ r\geq1}
\Pck_r(B_X(x,\Lambda r))<\infty.
\end{equation}
\end{definition}
Every doubling space, and every Euclidean metric subspace, satisfies
this condition. It guarantees finite admissible shell covers. The
metric-space definitions agree with the Euclidean definitions by
Lemma~\ref{lem:euclidean-intrinsic}.

\begin{theorem}\label{thm:mainG}
		Let $(X,d_X,o_X)$ and $(Y,d_Y,o_Y)$ be unbounded pointed metric spaces, both satisfying the uniform coarse packing condition. If $X$ and $Y$ are quasi-isometric, then, for every $\theta\in[0,1]$,
		\begin{equation}\label{eq:main-conclusion}
			\uDim_{\theta,o_X}X
			=
			\uDim_{\theta,o_Y}Y,
			\qquad
			\lDim_{\theta,o_X}X
			=
			\lDim_{\theta,o_Y}Y.
		\end{equation}
	\end{theorem}
	\begin{remark}[Choice of base point]\label{cor:basepoint}
Under the hypotheses on $X$, the dimensions are independent of $o$.
Indeed, apply Theorem~\ref{thm:mainG} to the surjective isometry
$\operatorname{id}_X:(X,d,o)\to(X,d,o')$. We may therefore omit the
base-point subscript. For Euclidean sets, the origin of the shells may
also be replaced by any $a\in\R^d$; see
Lemma~\ref{lem:euclidean-intrinsic}. Apply its same-radius recentering
argument to $E\cup\{a\}$ with base point $a$, then apply
Theorem~\ref{thm:mainG} to the inclusions obtained by adjoining $0$ and $a$.
\end{remark}
The main step in the proof is a content comparison uniform in the
auxiliary parameter. A quasi-isometry shifts shell indices by a bounded
amount; we adjust the scale and parameter together so that the transported
covers remain admissible (Lemma~\ref{lem:content-comparison}).

\subsection{Endpoints and continuity}\label{sec:endpoints-intro}
Write $\hdim E$ for Barlow--Taylor discrete Hausdorff dimension, recalled
in Section~\ref{sec:hausdorff}. For an unbounded locally finite set, put
\[
\dim_{\mathrm{LM}}E=\liminf_{n\to\infty}\frac{\log A_n(E)}{n\log2},
\qquad
\dim_{\mathrm{UM}}E=\limsup_{n\to\infty}\frac{\log A_n(E)}{n\log2},
\]
where $A_n(E)=\#(E\cap B(0,2^n))$. Locally finite means finite on every
bounded set. Uniformly discrete means that distinct points have distances
bounded below by one positive constant.
\begin{proposition}\label{gjH}
		Let $E\subset \R^d$. Then
		\[
		\u0dim E = \l0dim E = \hdim E.
		\]
	\end{proposition}
\begin{proposition}\label{dBD}
		Let $E\subset\R^d$ be an unbounded locally finite set.  Then
		\begin{equation}\label{eq:upper-mass-general-bound}
			\uodim E\leq\dim_{\mathrm{UM}}E.
		\end{equation}
		Moreover, if $E$ is uniformly discrete, then
		\begin{equation}\label{eq:upper-mass-endpoint}
			\uodim E=\dim_{\mathrm{UM}}E.
		\end{equation}
	\end{proposition}
The lower endpoint at $\theta=1$ need not equal lower mass dimension,
even for subsets of $\Z$ (Example~\ref{ex:lower-endpoint-mass-failure}).
Both dimension functions are nondecreasing and take values in $[0,d]$.
\begin{proposition}
		\label{thm:continuity}
		Let $E\subset\R^d$ and $0<\theta<\phi\le1$.  Then
		\begin{equation}
			0\le\lDim_\phi E-\lDim_\theta E
			\le\left(1-\frac{\theta}{\phi}\right)
			\left(d-\lDim_\theta E\right),
			\label{eq:lower-continuity-modulus}
		\end{equation}
		and
		\begin{equation*}
			0\le\uDim_\phi E-\uDim_\theta E
			\le\left(1-\frac{\theta}{\phi}\right)
			\left(d-\uDim_\theta E\right).
		\end{equation*}
		In particular, both spectra are continuous on $(0,1]$.
	\end{proposition}
Continuity at zero fails in general (Example~\ref{ex:discontinuity-at-zero}).

\subsection{Polynomial growth and sharp spectra}
Let $(X,d,o)$ be unbounded and suppose, for some $Q>0$ and
$v_-,v_+>0$, that
\begin{equation}\label{eq:Q-growth}
v_-R^Q\leq\#B_X(x,R)\leq v_+R^Q
\qquad(x\in X,\ R\geq1).
\end{equation}
These bounds imply uniform coarse packing for $X$ and its subspaces.
Mass dimensions are defined as above using $\#(E\cap B_X(o,2^n))$.
Write $S_n=S_n^X(o)$ in the following statements.
\begin{theorem}\label{thm:Q-growth}
Assume \eqref{eq:Q-growth}. Then, for every $\theta\in[0,1]$,
\begin{equation}\label{eq:Q-full-spectrum}
 \lDim_\theta X=\uDim_\theta X=Q.
\end{equation}
Every unbounded subspace $E\subset X$ satisfies
$0\leq\lDim_\theta E\leq\uDim_\theta E\leq Q$. For
$0<\theta<\phi\leq1$,
\begin{equation}\label{eq:Q-modulus}
 0\leq D_\phi E-D_\theta E
 \leq\left(1-\frac{\theta}{\phi}\right)(Q-D_\theta E),
\end{equation}
where $D$ denotes either the lower or the upper dimension.
\end{theorem}
\begin{theorem}\label{thm:Q-clusters}
Assume \eqref{eq:Q-growth}, fix $0<\beta<1$, and put $\alpha=1-\beta$.
Let $I\subset\N$ have bounded gaps, in the sense that
$I\cap[N,N+h]\ne\varnothing$ for all sufficiently large integers $N$
and some fixed $h$. Suppose that, for $n\in I$ sufficiently large,
there are $z_n\in X$ and nonempty sets $A_n$ satisfying
\begin{equation}\label{eq:Q-cluster-assumptions}
 A_n\subset S_n\cap B_X(z_n,C_0 2^{\beta n}),\qquad
 c_0 2^{Q\beta n}\leq\#A_n\leq C_1 2^{Q\beta n},
\end{equation}
where $C_0\geq1$ and $c_0,C_1>0$ are independent of $n$.
Set $E_\beta=\{o\}\cup\bigcup_{n\in I,\,n\geq n_0}A_n$, where $n_0$ is a sufficiently large integer for which \eqref{eq:Q-cluster-assumptions} holds. Then
\begin{equation}\label{eq:Q-cluster-spectrum}
 \lDim_\theta E_\beta=\uDim_\theta E_\beta=F_{\beta,Q}(\theta):=
 \begin{cases}
 0,&0\leq\theta\leq\alpha,\\[1mm]
 Q(1-\alpha/\theta),&\alpha<\theta\leq1.
 \end{cases}
\end{equation}
Both mass dimensions equal $Q\beta$, and for $\alpha\leq\theta<\phi\leq1$,
\begin{equation}\label{eq:Q-sharpness}
 F_{\beta,Q}(\phi)-F_{\beta,Q}(\theta)
 =\left(1-\frac{\theta}{\phi}\right)(Q-F_{\beta,Q}(\theta)).
\end{equation}
\end{theorem}
\begin{corollary}
\label{cor:Q-existence}
Every space satisfying \eqref{eq:Q-growth} contains, for each $0<\beta<1$,
a set with the spectrum and mass dimensions in
Theorem~\ref{thm:Q-clusters}. In particular, the modulus
\eqref{eq:Q-modulus} is sharp in each such space.
\end{corollary}
For clusters as in Theorem~\ref{thm:Q-clusters}, fix $0\leq\rho<Q\beta$.
Independently retain each $x\in A_n$ with probability $p_n=2^{-\rho n}$.
Equivalently, take independent Bernoulli variables $\xi_{n,x}$ with
$\mathbb P(\xi_{n,x}=1)=p_n$ on their countable product probability space.
Set
\begin{equation*}
A_n^\omega=\{x\in A_n:\xi_{n,x}(\omega)=1\},\qquad
E_{\beta,\rho}^\omega=\{o\}\cup\bigcup_{n\in I,\ n\geq n_0}A_n^\omega.
\end{equation*}
\begin{theorem}\label{thm:Q-thinning}
On a single event of probability one, simultaneously for every
$\theta\in[0,1]$,
\begin{equation}\label{eq:Q-random-spectrum}
 \lDim_\theta E_{\beta,\rho}^\omega
 =\uDim_\theta E_{\beta,\rho}^\omega=G_{\beta,\rho,Q}(\theta),
\end{equation}
where $G_{\beta,\rho,Q}(0)=0$ and, for $\theta>0$,
\[
 G_{\beta,\rho,Q}(\theta)
 =\min\left\{Q\beta-\rho,
 Q\left(1-\frac{1-\beta}{\theta}\right)_+\right\},
 \qquad t_+=\max\{t,0\}.
\]
Both mass dimensions of $E_{\beta,\rho}^\omega$ equal $Q\beta-\rho$.
If $0<\rho<Q\beta$, put
\begin{equation}\label{eq:Q-random-transition}
 \theta_* =\frac{Q(1-\beta)}{Q(1-\beta)+\rho}.
\end{equation}
Then $1-\beta<\theta_*<1$, and the spectrum is zero on $[0,1-\beta]$,
equals $Q(1-(1-\beta)/\theta)$ on $(1-\beta,\theta_*)$, and is
constant with value $Q\beta-\rho$ on $[\theta_*,1]$.
\end{theorem}
The word metric on the discrete Heisenberg group satisfies
\eqref{eq:Q-growth} with $Q=4$. Thus its full spectrum is constant at $4$,
and its sparse ball clusters realise both formulae above
(Corollary~\ref{thm:heisenberg-spectrum}). For the range $\mathcal R$ of
simple symmetric random walk on $\Z^d$, both spectra are almost surely
constant at $\min\{d,2\}$ (Corollary~\ref{cor:simple-random-walk}).

Section~\ref{sec:definitions} proves the endpoint and continuity results
and a mass distribution principle. Section~\ref{sec:lemmas} proves
quasi-isometric invariance. Section~\ref{sec:heisenberg} proves the growth
and spectrum formulae and gives the Heisenberg and random-walk applications.

\section{Endpoints and continuity}\label{sec:definitions}
	
	We first record the elementary content comparisons.
	\subsection{Monotonicity and range }

	\begin{lemma}\label{lem:theta-monotonicity}
		Let $E\subset\R^d$, $s\geq0$, $0<\delta<1$, and $0\leq\theta\leq\phi\leq1$.  Then
		\begin{equation}\label{contineq}
			\Phi_{\delta}^{s,\theta}(E)\leq \Phi_{\delta}^{s,\phi}(E).
		\end{equation}
		Consequently,
		\begin{equation}\label{dimineq}
			\lDim_{\theta}E\leq\lDim_{\phi}E,
			\qquad
			\uDim_{\theta}E\leq\uDim_{\phi}E.
		\end{equation}
	\end{lemma}
	
	\begin{proof}
		The interval $[0,\theta]$ is contained in $[0,\phi]$, so the content inequality \eqref{contineq} follows directly from \eqref{eq:Phi-delta}.  Passing to the critical exponents in Definition \ref{def:dimensions} gives the dimension inequalities \eqref{dimineq}.
	\end{proof}

	\begin{lemma}
		\label{lem:exponent-comparison}
		Let $E\subset\R^d$, $0\leq s<t$, $\eta\in[0,1]$, and
		$0<\delta<1$. Then
		\begin{equation}\label{lemma2.2_1}
			\Gamma_\delta^{t,\eta}(E)
			\leq
			\delta^{\eta(t-s)}\Gamma_\delta^{s,\eta}(E).
		\end{equation}
		In particular, for every $\theta\in[0,1]$,
		\begin{equation}\label{lemma2.2_2}
			\Phi_\delta^{t,\theta}(E)\leq\Phi_\delta^{s,\theta}(E).
		\end{equation}
	\end{lemma}
	
	\begin{proof}
		For each $n\in J_{\delta,\eta}$ and every radius occurring in an
		admissible cover of $E\cap S_n$, one has
		$0<r_{n,i}/2^n\leq\delta^\eta\leq1$. Therefore
		\[
		\left(\frac{r_{n,i}}{2^n}\right)^t
		\leq
		\delta^{\eta(t-s)}
		\left(\frac{r_{n,i}}{2^n}\right)^s.
		\]
		Taking the infimum over admissible covers for each
		$n\in J_{\delta,\eta}$ and then summing over $n$ proves \eqref{lemma2.2_1}. Since $\delta^{\eta(t-s)}\leq1$, taking the supremum over
		$0\leq\eta\leq\theta$ proves \eqref{lemma2.2_2}.
	\end{proof}
The admissible exponents in either dimension definition form an upper
set by Lemma~\ref{lem:exponent-comparison}. If its infimum is $D<\infty$
and $s>D$, there is an admissible $q<s$; exponent monotonicity then
makes $s$ admissible. We use this observation when passing from contents
to dimensions. The same argument applies to the intrinsic metric contents.

	\begin{proposition}
		\label{prop:basic-range}
		For every $E\subset\R^d$ and every $\theta\in[0,1]$,
		\begin{equation*}
			0\le\lDim_\theta E\le\uDim_\theta E\le d.
		\end{equation*}
	\end{proposition}
	
	\begin{proof}
We only need to show $\usdim E\le d.$  A maximal $1$-separated subset of $S_n$ has cardinality $O(2^{nd})$, and the corresponding unit balls cover $S_n$.  Hence, uniformly in $\eta\in[0,1]$,
\[
\Phi_{n,\delta}^{s,\eta}(E)\lesssim 2^{-n(s-d)}.
\]
Hence, for $s>d$,
\[
\Gamma_\delta^{s,\eta}(E)
\lesssim \sum_{n\ge N_\delta}2^{-n(s-d)}
\lesssim 2^{-N_\delta(s-d)},
\]
with constants independent of $\eta$.  Taking the supremum over $0\le\eta\le\theta$ gives the same estimate for $\Phi_\delta^{s,\theta}(E)$.  Since $2^{-N_\delta}\asymp\delta$,
\[
\limsup_{\delta\searrow0}
\frac{\log\Phi_\delta^{s,\theta}(E)}{-\log\delta}
\le -(s-d)<0.
\]
Thus $\uDim_\theta E\le s$ for every $s>d$, and letting $s\searrow d$ proves the result.
	\end{proof}

	\subsection{Continuity of the dimensions.}
	We prove Proposition~\ref{thm:continuity} by comparing covers at two parameters.
	\begin{lemma}
		\label{lem:shell-covering}
		There is a constant \(C_d\ge1\), depending only on \(d\), such that every
		set
		\[
		G\subset B(x,R)\cap S_n,
		\qquad
		R\ge u>0,
		\]
		can be covered by at most $C_d\left(1+\frac Ru\right)^d$
		balls of radius \(u\) whose centers belong to \(G\).
	\end{lemma}
	
	\begin{proof}
		Choose a maximal \(u\)-separated subset
		\(\{z_1,\ldots,z_M\}\subset G\).  Maximality implies that the
		\(u\)-balls centered at the \(z_j\) cover \(G\).  The balls
		\(B(z_j,u/2)\) are pairwise disjoint and lie in
		\(B(x,R+u/2)\).  Comparison of their \(d\)-dimensional volumes gives
		\[
		M
		\le
		\left(\frac{2R+u}{u}\right)^d
		\le
		3^d\left(1+\frac Ru\right)^d.
		\]
		Thus one may take \(C_d=3^d\).
	\end{proof}
	
	\begin{lemma}
		\label{lem:parameter-comparison}
		Let $E\subset\R^d$, $0<\theta<\phi\le1$, and $0\le s<d$.  Put
		\begin{equation*}
			T_{\theta,\phi}(s)
			=d-\frac{\theta}{\phi}(d-s).
		\end{equation*}
		Then
		\begin{equation*}
			\Gamma_\delta^{T_{\theta,\phi}(s),\phi}(E)
			\le C_d'\Gamma_\delta^{s,\theta}(E)
		\end{equation*}
		for every $0<\delta<1$, where $C_d'$ depends only on $d$.
	\end{lemma}
	
	\begin{proof}
Put $t=T_{\theta,\phi}(s)$ and $u_n=2^n\delta^\phi\geq1$.
For $n\in J_{\delta,\phi}\subset J_{\delta,\theta}$, take an admissible
$(\delta,\theta)$-cover. Retain balls of radius $r\leq u_n$; their
$t$-cost is at most their $s$-cost. For $r>u_n$, cover the intersection
with $E\cap S_n$ by at most $C_d'(r/u_n)^d$ balls of radius $u_n$,
using Lemma~\ref{lem:shell-covering}. Their total cost is at most
\[
C_d'(r/u_n)^d(u_n/2^n)^t
=C_d'(r/2^n)^s(r/2^n)^{d-s}\delta^{\phi(t-d)}
\leq C_d'(r/2^n)^s,
\]
since $r/2^n\leq\delta^\theta$ and
$\theta(d-s)+\phi(t-d)=0$. Taking infima and summing over
$J_{\delta,\phi}\subset J_{\delta,\theta}$ proves the claim.
\end{proof}

	\begin{proof}[Proof of Proposition~\ref{thm:continuity}]
		We prove the lower estimate; the upper one is identical with $\limsup$ in place of $\liminf$.  The left-hand inequality is Lemma \ref{lem:theta-monotonicity}.
		
		Assume $\lDim_\theta E<d$ and choose
		$\lDim_\theta E<s<d$.  Set
		\[
		t=T_{\theta,\phi}(s)=d-\frac{\theta}{\phi}(d-s).
		\]
		We claim that
		\begin{equation}\label{eq:regularized-parameter-comparison}
			\Phi_\delta^{t,\phi}(E)
			\le C_d'\Phi_\delta^{s,\theta}(E).
		\end{equation}
		Indeed, fix $0\le\eta\le\phi$.  If $\eta\le\theta$, then $t>s$ and Lemma \ref{lem:exponent-comparison} gives
		$\Gamma_\delta^{t,\eta}(E)\le\Gamma_\delta^{s,\eta}(E)\le\Phi_\delta^{s,\theta}(E)$.  If $\theta<\eta\le\phi$, put
		$t_\eta=T_{\theta,\eta}(s)$.  Then $t_\eta\le t$, so
		\[
		\Gamma_\delta^{t,\eta}(E)
		\le\Gamma_\delta^{t_\eta,\eta}(E)
		\le C_d'\Gamma_\delta^{s,\theta}(E)
		\le C_d'\Phi_\delta^{s,\theta}(E)
		\]
		by Lemma \ref{lem:parameter-comparison}.  Taking the supremum over $\eta$ proves \eqref{eq:regularized-parameter-comparison}.
		
		Since $s>\lDim_\theta E$, the upper-set observation after
Lemma~\ref{lem:exponent-comparison} gives
		\[
		\liminf_{\delta\searrow0}
		\frac{\log\Phi_\delta^{s,\theta}(E)}{-\log\delta}\le0.
		\]
		Equation \eqref{eq:regularized-parameter-comparison} therefore implies
		$\lDim_\phi E\le t$.  Letting $s\searrow\lDim_\theta E$ proves \eqref{eq:lower-continuity-modulus}.  The case $\lDim_\theta E=d$ follows from Proposition \ref{prop:basic-range}.  The proof for the upper dimension is the same.
	\end{proof}
	
	While the above proposition establishes the continuity of discrete intermediate dimensions on $(0, 1]$, the following example demonstrates that a discontinuity can occur at $\theta = 0$.
	
	\begin{example}[A jump at zero]\label{ex:discontinuity-at-zero}
For $n\geq4$, put $M_n=\lceil2^{n-\sqrt n}\rceil$ and
\[
E_n=\{2^{n-1},2^{n-1}+1,\ldots,2^{n-1}+M_n\}\subset S_n,
\qquad E=\bigcup_{n\geq4}E_n.
\]
Then $\lDim_0E=\uDim_0E=0$, whereas
$\lDim_\theta E=\uDim_\theta E=1$ for every $0<\theta\leq1$.
\begin{proof}
A ball of radius $2\cdot2^{n-\sqrt n}$ centered at the left endpoint
covers $E_n$. For $s>0$ and $N=N_\delta$ sufficiently large,
\[
\Phi_\delta^{s,0}(E)\lesssim\sum_{n\geq N}2^{-s\sqrt n}
\lesssim_s\sqrt N\,2^{-s\sqrt N}.
\]
The normalized logarithm has limsup at most zero, giving the zero endpoint.

Fix $0<\theta\leq1$ and $0\leq s<1$. At shell $N$, admissible radii
satisfy $1\leq r_i\leq u=2^N\delta^\theta$. Each ball covers at most
$3r_i$ integers, so $\sum_i r_i\geq(M_N+1)/3$. Therefore
\[
\sum_i(r_i/2^N)^s
\geq2^{-Ns}u^{s-1}\sum_i r_i
\gtrsim2^{-\sqrt N}\delta^{-\theta(1-s)}.
\]
Since $N\in J_{\delta,\theta}$, taking infima gives
\[
\liminf_{\delta\downarrow0}
\frac{\log\Phi_\delta^{s,\theta}(E)}{-\log\delta}
\geq\theta(1-s)>0.
\]
Thus $\lDim_\theta E\geq1$, and the range bound gives equality.
\end{proof}
\end{example}
	
	 A fixed exponential cluster size gives a continuous, nonconstant spectrum.
	\begin{example}\label{continuousspectrum}
		Let $0<\beta<1$ and write $\theta_0:=1-\beta.$ For a sufficiently large integer $N=N(\beta)$, define
		\[
		\text{$E_\beta=\bigcup_{n=N}^\infty E_n,$ where $E_n=\left\{2^{n-1}, 2^{n-1}+1,\ldots, 2^{n-1}+
			\left\lceil 2^{\beta n}\right\rceil\right\}$.}
		\]
		Then $\lDim_\theta E_\beta=\uDim_\theta E_\beta=f_\beta(\theta)$, where
		\[
		f_\beta(\theta)=
		\begin{cases}
			0, & 0\le \theta\le \theta_0,\\
			1-\frac{\theta_0}{\theta}, & \theta_0<\theta\le1.
		\end{cases}
		\]
		
		\begin{figure}[htbp]
\centering
\begin{tikzpicture}
\begin{axis}[
 width=10.5cm,height=6.2cm,
 axis lines=left,xmin=-0.025,xmax=1.07,ymin=-0.035,ymax=0.72,
 xlabel={$\theta$},ylabel={$f_\beta(\theta)$},
 xtick={0,0.4,1},xticklabels={$0$,$1-\beta=0.4$,$1$},
 ytick={0,0.6},yticklabels={$0$,$\beta=0.6$},
 tick label style={font=\small},label style={font=\small},
 clip=false]
\addplot[gray,densely dashed] coordinates {(0,0.6) (1,0.6) (1,0)};
\addplot[spectrumblue,very thick] coordinates {(0,0) (0.4,0)};
\addplot[spectrumblue,very thick,domain=0.4:1,samples=121] {1-0.4/x};
\addplot[only marks,mark=*,mark size=2pt,spectrumblue]
 coordinates {(0,0) (0.4,0) (1,0.6)};
\node[font=\small,anchor=south] at (axis cs:0.69,0.57)
 {$1-(1-\beta)/\theta$};
\end{axis}
\end{tikzpicture}
\caption{The spectrum in Example~\ref{continuousspectrum} for $\beta=0.6$.
It vanishes on $[0,1-\beta]$ and attains the continuity bound on its
increasing branch.}
\label{fig:euclidean-spectrum}
\end{figure}
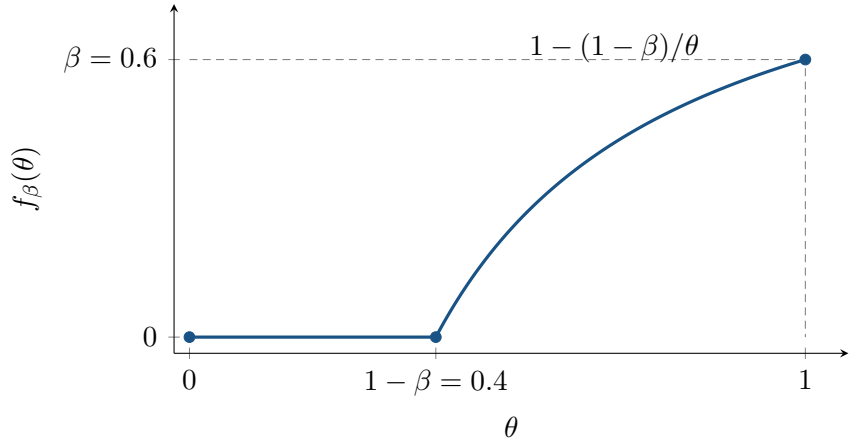
		
		\begin{proof}
Apply Theorem~\ref{thm:Q-clusters} with $X=\Z$, $Q=1$,
$A_n=E_n$, and all sufficiently large integers in $I$.
Indeed, $\#B_\Z(x,R)\asymp R$ for $R\geq1$,
$\#E_n\asymp2^{\beta n}$, and, for all sufficiently large $n$,
$E_n\subset S_n\cap B_\Z(2^{n-1},2^{\beta n+1})$. Thus the cluster
hypotheses hold with $C_0=2$.
Lemma~\ref{lem:euclidean-intrinsic} identifies the intrinsic and
Euclidean dimensions. Adjoining the base point is a quasi-isometry,
so Theorem~\ref{thm:mainG} justifies that step.
This proves the asserted formula.
\end{proof}
	\end{example}
	\begin{remark}\label{rem:sharp-modulus}
		The spectrum attains the upper bound in Proposition \ref{thm:continuity}.   More precisely, observe that $f_\beta(1)=\beta$ and $f_\beta(\theta_0)=0$.  For
		$\theta_0<\theta<\phi\le 1$,
		\[
		f_\beta(\phi)-f_\beta(\theta)=\theta_0\left(\frac{1}{\theta}-\frac{1}{\phi}\right)
		=\left(1-\frac{\theta}{\phi}\right)(1-f_\beta(\theta)).
		\]
		 Thus Example~\ref{continuousspectrum} shows that the modulus in Proposition~\ref{thm:continuity} is sharp.

A natural question is which functions arise as discrete intermediate spectra. Banaji and Rutar \cite{BR22} answered the corresponding question for bounded Euclidean sets.
	\end{remark}

	\subsection{The endpoint \texorpdfstring{$\theta=0$}{theta=0} and the discrete Hausdorff dimension}\label{sec:hausdorff}
	We use the ball-cover formulation of Barlow--Taylor discrete Hausdorff
dimension in \cite[Section 2.1, Definition 2.2]{DS}:
	\[
	\hdim E = \inf \left\{ s \geq 0 : \sum_{n=0}^\infty \Psi_{n,\infty}^{s}(E) < \infty \right\},
	\]
	where
	\[
	\begin{split}
\Psi_{n,\infty}^s(E)=\inf\biggl\{
&\sum_{i\in\mathcal I_n}(r_{n,i}/2^n)^s:
E\cap S_n\subseteq\bigcup_{i\in\mathcal I_n}B(x_i,r_{n,i}),\\
&x_i\in S_n,\quad r_{n,i}\geq1,
\quad\mathcal I_n\text{ finite}\biggr\}.
\end{split}
\]
	For more properties  of the discrete Hausdorff dimension, see \cite{BT89,BT,DS}.
	
	The radius cap in the zero-parameter content does not change this critical exponent.
\begin{lemma}\label{kbj}
		Let $E\subset \R^d$ be  an unbounded set. Fix $s\ge 0$ and $\delta\in (0,1).$ There exists a constant $C=C(d,s)$ such that
		\[
		\Psi_{n,\infty}^{s}(E)\le \Phi_{n,\delta}^{s,0}(E)\le C\Psi_{n,\infty}^{s}(E)
		\]
		for any $n\ge N_\delta$.
	\end{lemma}
	\begin{proof}
The first inequality follows by inclusion of the admissible covers.
For the reverse inequality, take a cover in the definition of
$\Psi_{n,\infty}^s(E)$. If all radii are at most $2^n$, it is already
admissible. Otherwise its cost is at least one, whereas
Lemma~\ref{lem:shell-covering} covers the whole shell by at most $M_d$
balls of radius $2^n$ with centers in the shell. The capped content is
therefore at most $\max\{1,M_d\}$ times the original cost. Take infima.
\end{proof}

	We can now identify the zero endpoint.
	
	\begin{proof}[Proof of Proposition~\ref{gjH}]
Bounded sets have both dimensions zero, so suppose $E$ is unbounded.
At $\eta=0$, each shell content is independent of $\delta$.
By Lemma~\ref{kbj}, its series converges precisely when the corresponding
Hausdorff shell series converges. For the capped contents the convergence
exponents form an upper set, since $r_i/2^n\leq1$. Thus if $s>\hdim E$, the tails tend to zero,
so the normalized logarithmic content has limsup at most zero. If
$s<\hdim E$, every tail is infinite, so this exponent is inadmissible
for either dimension. The two critical exponents are therefore $\hdim E$.
\end{proof}
	
	\subsection{The endpoint \texorpdfstring{$\theta=1$}{theta=1} and the discrete box dimensions}\label{sec:mass}
	We call $\lDim_1E$ and $\uDim_1E$ the lower and upper discrete
box-type dimensions. They are the endpoints of the regularized family,
and should not be identified with the entropy indices of \cite{BT}.
Recall the mass dimensions defined in Section~1. Put
$a_n(E)=\#(E\cap S_n)$.

For the reverse inequality under the uniform discreteness assumption,
	we  need to record the following three elementary facts.
	
	\begin{lemma}\label{lem:bounded-balls}
		Let $E\subset\R^d$ be uniformly discrete.  There is a constant $M<\infty$ such that
		\[
		\#(E\cap B(x,r))\le M
		\]
		for every $x\in\R^d$ and every $1\le r<2$.
	\end{lemma}
	
	\begin{proof}
If $c>0$ is a separation constant, the disjoint balls of radius
$c/2$ around points of $E\cap B(x,r)$ lie in $B(x,r+c/2)$.
Volume comparison gives $\#(E\cap B(x,r))\leq(1+4/c)^d$ for $1\leq r<2$.
\end{proof}

	\begin{lemma}\label{lem:local-equivalence}
		Let $E\subset\R^d$ be uniformly discrete, $s\ge0$, $0<\delta<1$, and put $n=N_\delta$.  Then
		\begin{equation*}
			M^{-1}a_n(E)2^{-ns}
			\le \Gamma_\delta^{s,1}(E)
			=\Phi_{N_\delta,\delta}^{s,1}(E)
			\le a_n(E)2^{-ns}.
		\end{equation*}
	\end{lemma}
	
	\begin{proof}
For $n=N_\delta$, the index set $J_{\delta,1}$ contains only $n$,
and $1\leq r_i\leq2^n\delta<2$. Unit balls centered at all points of
$E\cap S_n$ give the upper bound. Each admissible ball contains at most
$M$ points by Lemma~\ref{lem:bounded-balls}, and has cost at least
$2^{-ns}$, giving the lower bound.
\end{proof}
	
	\begin{lemma}\label{lem:shell-upper-mass}
		Let $E\subset\R^d$ be an unbounded uniformly discrete set.  Then
		\begin{equation}\label{eq:shell-upper-mass}
			\limsup_{n\to\infty}
			\frac{\log(1+a_n(E))}{n\log2}
			=\dim_{\mathrm{UM}}E.
		\end{equation}
	\end{lemma}
	
	\begin{proof}
		Since $E$ is unbounded, $A_n(E)\geq1$ for all sufficiently large $n$.
		As $a_n(E)\leq A_n(E)$, we then have
		\[
		\log(1+a_n(E))\leq\log(1+A_n(E))\leq\log2+\log A_n(E).
		\]
		It follows that the left-hand side of \eqref{eq:shell-upper-mass} is at
		most $\dim_{\mathrm{UM}}E$.  Conversely, let the left-hand side be $\beta$ and fix $\varepsilon>0$.  For all sufficiently large $j$,
		$a_j(E)\le2^{(\beta+\varepsilon)j}$.  Hence
		\[
		A_n(E)\le A_{j_0-1}(E)+\sum_{j=j_0}^n2^{(\beta+\varepsilon)j}
		\lesssim2^{(\beta+\varepsilon)n}.
		\]
		Thus $\dim_{\mathrm{UM}}E\le\beta+\varepsilon$, and \eqref{eq:shell-upper-mass} follows by letting $\varepsilon\searrow0$.
	\end{proof}

	\begin{proof}[Proof of Proposition~\ref{dBD}]
Write $q=\dim_{\mathrm{UM}}E$. If $q=+\infty$ there is nothing to prove
for the upper bound. Otherwise, for $t>q$, choose $\varepsilon>0$ with
$t>q+\varepsilon$. At all sufficiently large shells,
$a_n(E)\leq2^{(q+\varepsilon)n}$. Unit-ball covers are admissible for
all auxiliary parameters, so
\[
\Phi_\delta^{t,1}(E)\leq\sum_{n\geq N_\delta}
a_n(E)2^{-nt}\lesssim\delta^{t-q-\varepsilon}.
\]
Hence $\uDim_1E\leq q$.

For uniformly discrete $E$, disjoint balls of a fixed radius give
$A_n(E)\lesssim2^{nd}$, so $q\leq d<\infty$. If $q>0$, take $0\leq u<q$.
At $\delta_n=2^{-n}$, Lemma~\ref{lem:local-equivalence} gives
$\Phi_{\delta_n}^{u,1}(E)\geq M^{-1}a_n(E)2^{-nu}$.
Lemma~\ref{lem:shell-upper-mass} implies
$\limsup_n\log a_n(E)/(n\log2)=q$ when $q>0$:
on a subsequence realizing this positive rate, $a_n(E)\geq1$ eventually
and $0\leq\log(1+a_n(E))-\log a_n(E)\leq\log2$. Thus the normalized
logarithmic content has limsup at least $q-u>0$, proving
$\uDim_1E\geq q$. The case $q=0$ follows from nonnegativity.
\end{proof}

	The next example shows that \eqref{eq:upper-mass-general-bound} can be strict and, consequently,
	that equality in \eqref{eq:upper-mass-endpoint} can fail without uniform discreteness.
	\begin{example}
		\label{ex:uniform-discreteness-necessary}
		For each $n\geq1$, define
		\[
		E_n
		=
		\left\{
		3\cdot2^{n-2}+j\,2^{-3n}:0\leq j<2^n
		\right\},
		\qquad
		E=\bigcup_{n\geq1}E_n\subset\R.
		\]
		Then $E\cap S_n=E_n$, $\#E_n=2^n$, and
$A_N(E)=2^{N+1}-2$, so both mass dimensions are $1$.
The set is locally finite but not uniformly discrete. Since
$\diam E_n<1$, each cluster has an admissible unit-ball cover. For $s>0$,
\[
\Phi_\delta^{s,1}(E)\leq\sum_{n\geq N_\delta}2^{-ns}
\lesssim_s\delta^s.
\]
Thus $\uDim_1E=0<1=\dim_{\mathrm{UM}}E$.
\end{example}

	The following example shows that the analogue of
	\eqref{eq:upper-mass-general-bound} obtained by replacing both upper
	dimensions with their lower counterparts fails in general, even for
	uniformly discrete sets.
	
	\begin{example}[Failure of the lower mass identity]\label{ex:lower-endpoint-mass-failure}
Put $m_k=k!$ for $k\geq2$ and
$E=\bigcup_{k\geq2}(S_{m_k}\cap\Z)$. Then
$\dim_{\mathrm{LM}}E=0$ but $\lDim_1E=1$.
\begin{proof}
At $n_k=m_{k+1}-1$,
\[
A_{n_k}(E)=\sum_{j=2}^k2^{m_j}\leq k2^{m_k},\qquad
\frac{\log A_{n_k}(E)}{n_k\log2}
\leq\frac{m_k+\log_2k}{m_{k+1}-1}\longrightarrow0.
\]
This proves the lower mass claim. For $0\leq s<1$, an $\eta=0$
cover of an occupied shell $S_m\cap\Z$ has
$1\leq r_i\leq2^m$ and $2^m\leq3\sum_i r_i$. Hence
\[
\sum_i(r_i/2^m)^s\geq2^{-m}\sum_i r_i\geq\tfrac13.
\]
There are infinitely many occupied shells. Hence every tail
$\Gamma_\delta^{s,0}(E)$ is infinite, as is $\Phi_\delta^{s,1}(E)$. Thus $\lDim_1E\geq1$; the range
bound gives equality. In fact, both spectra are identically one on $[0,1]$.
\end{proof}
\end{example}

	\subsection{A mass distribution principle}
The following criterion gives lower bounds by distributing mass across
admissible shells, in the spirit of the Hausdorff mass distribution
principle \cite{Fal}.
\begin{proposition}[Mass distribution principle]\label{Mmdp}
		Let $E\subset\R^d$ be an unbounded Borel set, let $s>0$, and let
		$0\leq\theta\leq1$.  Suppose that there are constants $c_0,c_1>0$
		and $\delta_0\in(0,1)$ such that, for every
		$0<\delta<\delta_0$, there is a finite Borel measure $\mu_\delta$
		supported on
		$
		E\cap\bigcup_{n\in J_{\delta,\theta}}S_n
		$
		with
		\begin{equation*}
			\mu_\delta\left(E\cap\bigcup_{n\in J_{\delta,\theta}}S_n\right)\geq c_0
		\end{equation*}
		and
		\begin{equation}\label{eq:mass-local}
			\mu_\delta(E\cap S_n\cap B(x,r))
			\leq c_1\left(\frac r{2^n}\right)^s
		\end{equation}
		for every $n\in J_{\delta,\theta}$, $x\in S_n$, and
		$1\leq r\leq2^n\delta^\theta$.  Then
		\begin{equation*}
			\lDim_\theta E\geq s.
		\end{equation*}
		If the same hypotheses hold only along a sequence
		$\delta_k\searrow0$ in $(0,1)$, with $c_0$ and $c_1$ independent of
		$k$, then
		\begin{equation*}
			\uDim_\theta E\geq s.
		\end{equation*}
	\end{proposition}
	
	\begin{proof}
For any scale at which the measure is available, apply
\eqref{eq:mass-local} to an admissible cover in each shell. Taking
infima and summing, using the disjoint shells and the total-mass
assumption, gives
\begin{equation}\label{eq:mass-gamma-lower}
\Gamma_\delta^{s,\theta}(E)\geq c_0/c_1.
\end{equation}
If $\theta>0$ and $0\leq t<s$, Lemma~\ref{lem:exponent-comparison}
therefore yields
\[
\Phi_\delta^{t,\theta}(E)\geq\Gamma_\delta^{t,\theta}(E)
\geq(c_0/c_1)\delta^{-\theta(s-t)}.
\]
Taking the normalized liminf gives $\lDim_\theta E\geq s$ when
measures exist at all small scales. Along a sequence of scales the same
estimate gives $\uDim_\theta E\geq s$.

For $\theta=0$, the left-hand side of \eqref{eq:mass-gamma-lower}
is a tail of a fixed nonnegative series with finite terms. A convergent
series has tails tending to zero, contradicting that bound even along
a sequence. The series therefore diverges and every tail is infinite.
By exponent monotonicity, all $t<s$ are inadmissible for both dimensions.
\end{proof}
	
	At $\theta=0$, the argument is a tail-series form of the macroscopic
mass distribution method; compare \cite[Proposition 2.5]{DS}.

\section{Quasi-isometric invariance}\label{sec:lemmas}
	We use the metric contents and packing condition defined in
Section~1. As in Lemma~\ref{lem:exponent-comparison}, for $t\geq s$,
\begin{equation*}
\Phi_{n,\delta,o}^{t,\eta}(X)\leq\Phi_{n,\delta,o}^{s,\eta}(X),\qquad
\Gamma_{\delta,o}^{t,\eta}(X)\leq\Gamma_{\delta,o}^{s,\eta}(X).
\end{equation*}
\subsection{Packing and intrinsic covers}
\begin{remark}\label{rem:regular-tree-ucp}
Uniform bounds on the cardinalities of balls of each fixed radius do not imply uniform coarse packing. For example, let $T$ be a rooted binary tree with its
vertex path metric. For an integer $k\geq1$, choose one descendant at level
$2k$ of each vertex at level $k$. These $2^k$ vertices belong to
$B_T(o,3k)$ and have pairwise distances at least $2k$. Consequently,
$\Pck_k(B_T(o,3k))\geq2^k$, so $T$ fails \eqref{eq:UCP}, although its
vertex degrees are uniformly bounded. Theorem~\ref{thm:mainG} therefore
makes no assertion for all bounded-degree trees or all hyperbolic graphs.
\end{remark}

	The following lemma is a generalization of Lemma \ref{lem:shell-covering}.
	\begin{lemma}\label{lem:internal-cover}
		
		Suppose that $X$ satisfies the uniform coarse packing condition. Let $x\in X$, $\Lambda\geq1$, $r\geq1$, and $G\subset B_X(x,\Lambda r)$. Then $G$ can be covered by at most $M_X(\Lambda)$ balls of radius $r$ whose centers belong to $G$.
	\end{lemma}
	
	\begin{proof}
		Choose a maximal $r$-separated subset $Z$ of $G$. Since $Z\subset B_X(x,\Lambda r)$, condition \eqref{eq:UCP} gives $\card Z\leq M_X(\Lambda)$. If some $y\in G$ did not belong to $\bigcup_{z\in Z}B_X(z,r)$, then $d(y,z)\geq r$ for every $z\in Z$, so $y$ could be adjoined to $Z$, contradicting maximality. Thus the indicated balls cover $G$.
	\end{proof}
	
	The following proposition shows that the standard doubling condition implies uniform coarse packing. Recall that a metric space is called doubling if there exists an integer $N_D\geq1$ such that every ball of radius $R$ can be covered by at most $N_D$ balls of radius $R/2$, for every $R>0$.
	\begin{proposition}\label{prop:doubling-implies-UCP}
		Suppose that $X$ is doubling. Then $X$ satisfies the uniform coarse packing condition.
	\end{proposition}
	
	\begin{proof}
		Fix $\Lambda\geq1$, and choose an integer $k$ such that $\Lambda2^{-k}<1/2$. By iterating the doubling cover, the ball $B_X(x,\Lambda r)$ can be covered by at most $N_D^k$ balls of radius strictly smaller than $r/2$. Each such ball contains at most one point of an $r$-separated set, since its diameter is strictly smaller than $r$. Therefore
		\[
		\Pck_r(B_X(x,\Lambda r))\leq N_D^k
		\]
		for all $x\in X$ and all $r\geq1$.
	\end{proof}

	Clearly, Euclidean spaces are unbounded pointed metric spaces satisfying the uniform coarse packing condition. The following lemma guarantees that  Theorem \ref{thm:mainG} implies Theorem \ref{thm:main}.
	\begin{lemma}\label{lem:euclidean-intrinsic}
		Let $E\subset\R^d$ be unbounded, equipped with the induced Euclidean
		metric, and choose $o\in E$. Then, for every $\theta\in[0,1]$,
		\[
		\uDim_\theta E=\uDim_{\theta,o}E,
		\qquad
		\lDim_\theta E=\lDim_{\theta,o}E,
		\]
		where the dimensions on the left are those of
		Definitions~\ref{def:Phi}--\ref{def:dimensions}, whose shells are
		centered at $0\in\R^d$, and those on the right are the pointed
		metric-space dimensions of Definition~\ref{def:metric-dimensions}.
	\end{lemma}

	\begin{proof}
Put $\widetilde E=E\cup\{0\}$ with base point zero, so that
$S_n^{\widetilde E}(0)=E\cap S_n$ for $n\geq1$.
An intrinsic cover is an ambient cover. Conversely, for each ball
$B(x_i,r_i)$ in an ambient cover, Lemma~\ref{lem:shell-covering}
covers $E\cap S_n\cap B(x_i,r_i)$ by at most $C_d2^d$ balls of the
same radius $r_i$ with centers in that intersection. Hence
\begin{equation*}
\Phi_{n,\delta}^{s,\eta}(E)
\leq\Phi_{n,\delta,0}^{s,\eta}(\widetilde E)
\leq C_d2^d\Phi_{n,\delta}^{s,\eta}(E).
\end{equation*}
Summing and taking the supremum identifies the ambient dimensions of
$E$ with the intrinsic dimensions of $(\widetilde E,0)$.
The inclusion $E\hookrightarrow\widetilde E$ is a quasi-isometry:
the added point lies within $|o|$ of $o\in E$. Euclidean subspaces
satisfy uniform coarse packing, so Theorem~\ref{thm:mainG} completes
the identification. Its proof below uses only intrinsic contents and
is independent of this lemma.
\end{proof}

	\subsection{\texorpdfstring{Proof of Theorem~\ref{thm:mainG}}{Proof of Theorem \ref{thm:mainG}}}
	To prove Theorem~\ref{thm:mainG}, we need the following three lemmas.
	\begin{lemma}\label{lem:shell-localization}
		Let $f:X\to Y$ be an $(L,c,C)$-quasi-isometry between unbounded
pointed metric spaces $(X,d_X,o_X)$ and $(Y,d_Y,o_Y)$. There exist integers $K\geq1$ and $\ell_0\geq1$ such that, for every $\ell\geq\ell_0$ and every $y\in S_\ell^Y(o_Y)$, one can choose $p(y)\in X$ satisfying
		\begin{equation}\label{eq:preimage-choice}
			d_Y(y,f(p(y)))\leq C
		\end{equation}
		and
		\begin{equation*}
			p(y)\in
			\bigcup_{|n-\ell|\leq K}S_n^X(o_X).
		\end{equation*}
	\end{lemma}
	
	\begin{proof}
		By \eqref{eq:qi-density}, for any $y\in Y$ we can choose $p(y)$ satisfying \eqref{eq:preimage-choice}. Put $y_0=f(o_X)$ and $A=d_Y(o_Y,y_0)$. If $y\in S_\ell^Y(o_Y)$, then $
		2^{\ell-1}\leq d_Y(o_Y,y)<2^\ell.$
		The triangle inequality gives
		\[
		d_Y(f(p(y)),y_0)
		\geq2^{\ell-1}-C-A
		\]
		and
		\[
		d_Y(f(p(y)),y_0)
		\leq2^\ell+C+A.
		\]
		Substituting these two estimates into the upper and lower quasi-isometry bounds, respectively, yields
		\[
		L^{-1}\bigl(2^{\ell-1}-C-A-c\bigr)
		\leq d_X(p(y),o_X)
		\leq L\bigl(2^\ell+C+A+c\bigr).
		\]
		Consequently, for all sufficiently large $\ell$, there exist constants $a,b>0$, independent of $\ell$ and $y$, such that
		\[
		a2^\ell\leq d_X(p(y),o_X)\leq b2^\ell.
		\]
		Choose $K$ so large that every number in $[a2^\ell,b2^\ell]$ lies in a dyadic shell whose index $n$ satisfies $|n-\ell|\leq K$. This proves the assertion.
	\end{proof}

	\begin{lemma}\label{lem:content-comparison}
		Fix $s\geq0$ and $\theta\in[0,1]$. Let $f:X\to Y$ be an $(L,c,C)$-quasi-isometry, and suppose that $Y$ satisfies the uniform coarse packing condition. Then there exist a constant $A\geq1$, an integer $K_0\geq1$, and $\delta_0>0$ such that
		\begin{equation}\label{eq:content-comparison}
			\Phi_{\delta,o_Y}^{s,\theta}(Y)
			\leq
			A\Phi_{2^{K_0}\delta,o_X}^{s,\theta}(X)
		\end{equation}
		whenever $0<\delta<\delta_0$.
	\end{lemma}
	
	\begin{proof}
		Let $K$ and $\ell_0$ be supplied by Lemma~\ref{lem:shell-localization}. Enlarge $K$ and reduce $\delta_0$, if necessary, so that for $0<\delta<\delta_0$,
		$
		2^K\delta<1,
		\quad
		N_\delta>2K+\ell_0.
		$
		Set
		\[
		N=N_\delta,
		\quad
		\delta'=2^K\delta,
		\quad
		N_{\delta'}=N-K.
		\]
		
		Fix $\eta\in[0,\theta]$. If $\eta=0$, set $\eta'=0$. If $\eta>0$, define
		\begin{equation*}
			M_\eta=\left\lceil\frac{N}{\eta}\right\rceil+K,
			\qquad
			\eta'=\frac{N-K}{M_\eta}.
		\end{equation*}
		For $\eta>0$, we have $0<\eta'\leq\eta$, and
		\[
		J_{\delta',\eta'}
		=
		\{N-K,N-K+1,\ldots,M_\eta\}.
		\]
		Thus, whenever $\ell\in J_{\delta,\eta}$ and $|n-\ell|\leq K$, one has $n\in J_{\delta',\eta'}$.
		
		We also need a uniform estimate for the parameter displacement. From
		\[
		0\leq
		\eta\left\lceil\frac{N}{\eta}\right\rceil-N
		<\eta
		\]
		and $N/M_\eta\leq\eta$, it follows that
		\begin{align}
			N(\eta-\eta')
			&=
			\frac{N}{M_\eta}
			\bigl(\eta M_\eta-N+K\bigr)\notag\\
			&<
			\eta\bigl(\eta+\eta K+K\bigr)
			\leq2K+1.
			\label{eq:eta-shift}
		\end{align}
		For $\eta=0$, this estimate is immediate.
		
		For each $n\in J_{\delta',\eta'}$, choose an arbitrary $(\delta',\eta')$-admissible cover of $S_n^X(o_X)$,
		\begin{equation}\label{eq:source-cover}
			S_n^X(o_X)
			\subseteq
			\bigcup_{i\in I_n}B_X(x_{n,i},r_{n,i}),
			\qquad
			1\leq r_{n,i}\leq2^n(\delta')^{\eta'}.
		\end{equation}
		If no such finite cover exists for some relevant $n$, then the right-hand side of \eqref{eq:content-comparison} is infinite and there is nothing to prove. Hence we may work with finite admissible covers.
		
		Fix $\ell\in J_{\delta,\eta}$. For $|n-\ell|\leq K$ and $i\in I_n$, put
		\[
		T_{\ell,n,i}
		=
		\left\{y\in S_\ell^Y(o_Y):
		p(y)\in S_n^X(o_X)\cap B_X(x_{n,i},r_{n,i})
		\right\}.
		\]
		If  $T_{\ell,n,i}\neq\emptyset$, choose $y_{\ell,n,i}\in T_{\ell,n,i}$. For any $y\in T_{\ell,n,i}$, equations \eqref{eq:preimage-choice} and \eqref{eq:qi-distortion}, together with $r_{n,i}\geq1$, give
		\begin{align*}
			d_Y(y,y_{\ell,n,i})
			&\leq
			2C+d_Y(f(p(y)),f(p(y_{\ell,n,i})))\\
			&\leq2C+2Lr_{n,i}+c
			<Dr_{n,i},
		\end{align*}
		where $
		D=2C+2L+c+1.
		$
		
		Let $u_\ell=2^\ell\delta^\eta $
		be the largest radius allowed on the target shell. Since $\ell\geq N$, \eqref{eq:Ndelta} and $0\leq\eta\leq1$ imply
		\[
		u_\ell\geq2^N\delta^\eta\geq2^{N(1-\eta)}\geq1.
		\]
		We treat each $T_{\ell,n,i}$ in two cases.
		
		\medskip
		\noindent
		\textbf{Case 1: $Dr_{n,i}\leq u_\ell$.}
		A single ball $B_Y(y_{\ell,n,i},Dr_{n,i})$ is admissible. Since $|n-\ell|\leq K$,
		\begin{equation}\label{eq:case-one-cost}
			\Phi_{\ell,\delta,o_Y}^{s,\eta}(T_{\ell,n,i})
			\leq
			\left(\frac{Dr_{n,i}}{2^\ell}\right)^s
			\leq
			D^s2^{Ks}\left(\frac{r_{n,i}}{2^n}\right)^s.
		\end{equation}
		
		\medskip
		\noindent
		\textbf{Case 2: $Dr_{n,i}>u_\ell$.}
		By \eqref{eq:source-cover}, $|n-\ell|\leq K$, and \eqref{eq:eta-shift}, if $t=-\log_2\delta$, then $N-1<t\leq N$, and hence
		\begin{align*}
			\frac{Dr_{n,i}}{u_\ell}
			&\leq
			D2^{n-\ell}2^{K\eta'}\delta^{\eta'-\eta}\\
			&\leq
			D2^K2^K2^{N(\eta-\eta')}\\
			&\leq D2^{4K+1}
			< D_0,
		\end{align*}
		where one may take $D_0=D2^{4K+2}$. Therefore
		\[
		T_{\ell,n,i}
		\subset B_Y(y_{\ell,n,i},D_0u_\ell).
		\]
		By Lemma~\ref{lem:internal-cover}, the set $T_{\ell,n,i}$ can be covered by at most $M_Y(D_0)$ balls of radius $u_\ell$ whose centers belong to $T_{\ell,n,i}\subset S_\ell^Y(o_Y)$. Since $u_\ell<Dr_{n,i}$,
		\begin{equation}\label{eq:case-two-cost}
			\Phi_{\ell,\delta,o_Y}^{s,\eta}(T_{\ell,n,i})
			\leq
			M_Y(D_0)
			\left(\frac{u_\ell}{2^\ell}\right)^s
			\leq
			M_Y(D_0)(D2^K)^s
			\left(\frac{r_{n,i}}{2^n}\right)^s.
		\end{equation}
		
		By shell localization, every $y\in S_\ell^Y(o_Y)$ belongs to one of the sets $T_{\ell,n,i}$. Combining \eqref{eq:case-one-cost}, \eqref{eq:case-two-cost}, and finite subadditivity, we obtain a constant $C_0$, depending only on $s$, the quasi-isometry constants, $K$, and the coarse packing function of $Y$, such that
		\[
		\Phi_{\ell,\delta,o_Y}^{s,\eta}(Y)
		\leq
		C_0
		\sum_{|n-\ell|\leq K}
		\sum_{i\in I_n}
		\left(\frac{r_{n,i}}{2^n}\right)^s.
		\]
		For each fixed $\ell$, only finitely many source shells occur.
Taking covers within an arbitrarily small error of each infimum gives
		\begin{equation*}
			\Phi_{\ell,\delta,o_Y}^{s,\eta}(Y)
			\leq
			C_0
			\sum_{|n-\ell|\leq K}
			\Phi_{n,\delta',o_X}^{s,\eta'}(X).
		\end{equation*}
		For $\eta=0$, the following summation is justified by nonnegativity,
including when either series is infinite. Summing over $\ell\in J_{\delta,\eta}$, and observing that each source shell occurs at most $2K+1$ times, yields
		\begin{align*}
			\Gamma_{\delta,o_Y}^{s,\eta}(Y)
			&\leq
			C_0(2K+1)
			\Gamma_{\delta',o_X}^{s,\eta'}(X)\\
			&\leq
			C_0(2K+1)
			\Phi_{\delta',o_X}^{s,\theta}(X),
		\end{align*}
		where the last inequality uses $\eta'\leq\eta\leq\theta$. Taking the supremum over $0\leq\eta\leq\theta$ proves \eqref{eq:content-comparison}. One may take $K_0=K$ and $A=C_0(2K+1)$.
	\end{proof}

	The following lemma shows that every quasi-isometry has a coarse inverse.
	\begin{lemma}\label{lem:quasi-inverse}
		Every quasi-isometry $f:X\to Y$ admits a quasi-isometry $g:Y\to X$ such that $g\circ f$ is at uniformly bounded distance from $\operatorname{id}_X$ and $f\circ g$ is at uniformly bounded distance from $\operatorname{id}_Y$.
	\end{lemma}
	
	\begin{proof}
Choose $g(y)$ with $d_Y(y,f(g(y)))\leq C$. The triangle inequality
and the distortion bounds for $f$ give
\[
L^{-1}d_Y(y,y')-L^{-1}(2C+c)
\leq d_X(g(y),g(y'))
\leq Ld_Y(y,y')+L(2C+c).
\]
Also $d_X(g(f(x)),x)\leq L(C+c)$, so $g$ is coarsely surjective.
These estimates give all the assertions.
\end{proof}

	\begin{proof}[Proof of Theorem~\ref{thm:mainG}]
		Fix $s\geq0$ and $\theta\in[0,1]$. By Lemma~\ref{lem:content-comparison}, there exist $A\geq1$ and $\beta>1$ such that
		\begin{equation}\label{eq:one-way}
			\Phi_{\delta,o_Y}^{s,\theta}(Y)
			\leq
			A\Phi_{\beta\delta,o_X}^{s,\theta}(X)
		\end{equation}
		for all sufficiently small $\delta$.
		
		Suppose that $s>\lDim_{\theta,o_X}X$. By the definition and monotonicity in $s$, there exists a sequence $\varepsilon_k\searrow0$ such that
		\[
		\liminf_{k\to\infty}
		\frac{\log\Phi_{\varepsilon_k,o_X}^{s,\theta}(X)}{-\log\varepsilon_k}
		\leq0.
		\]
		Put $\delta_k=\varepsilon_k/\beta$. From \eqref{eq:one-way},
		\[
		\frac{\log\Phi_{\delta_k,o_Y}^{s,\theta}(Y)}{-\log\delta_k}
		\leq
		\frac{\log A+
			\log\Phi_{\varepsilon_k,o_X}^{s,\theta}(X)}
		{-\log(\varepsilon_k/\beta)}.
		\]
		The denominator ratio
$(-\log\varepsilon_k)/(-\log(\varepsilon_k/\beta))$ tends to one,
and $\log A/(-\log(\varepsilon_k/\beta))$ tends to zero. Hence $ \lDim_{\theta,o_Y}Y\leq s.$
		Letting $s\searrow\lDim_{\theta,o_X}X$ gives
		\[
		\lDim_{\theta,o_Y}Y
		\leq
		\lDim_{\theta,o_X}X.
		\]
		
		Similarly, if $s>\uDim_{\theta,o_X}X$, then \eqref{eq:one-way} holds at all sufficiently small scales, and taking a $\limsup$ directly yields
		\[
		\uDim_{\theta,o_Y}Y
		\leq
		\uDim_{\theta,o_X}X.
		\]
		
		By Lemma~\ref{lem:quasi-inverse}, $f$ has a quasi-isometric coarse inverse $g:Y\to X$. Since $X$ also satisfies the uniform coarse packing condition, applying the preceding argument to $g$ gives both reverse inequalities. If a source dimension is infinite, the corresponding one-way inequality is automatic; applying the finite-source argument in the opposite direction excludes unequal finite and infinite values. Therefore \eqref{eq:main-conclusion} holds.
	\end{proof}

\begin{proof}[Proof of Theorem~\ref{thm:main}]
Euclidean metric subspaces satisfy uniform coarse packing, by the
packing argument of Lemma~\ref{lem:shell-covering}.
Choose $o_E\in E$ and $o_F\in F$ and apply Theorem~\ref{thm:mainG}
to these pointed spaces. Lemma~\ref{lem:euclidean-intrinsic} identifies
both intrinsic dimensions with their ambient Euclidean counterparts.
\end{proof}

\section{Uniform polynomial growth and probabilistic applications}
\label{sec:heisenberg}

We prove the polynomial-growth and spectrum results stated in Section~1,
then apply them to the Heisenberg group and random-walk ranges.
All subspaces carry their induced metrics, and all balls are open.

\subsection{Uniform growth and the continuity modulus}
Throughout this section assume \eqref{eq:Q-growth}. Bounded sets are
finite, and all subspaces carry their induced metrics. We use the
intrinsic contents from Section~1. If $o\notin E$, add $o$; the inclusion
$E\hookrightarrow E\cup\{o\}$ is a quasi-isometry, so this does not
change the dimensions once Lemma~\ref{lem:Q-packing} is established.

We first derive uniform covering and packing estimates from the volume growth assumption.
\begin{lemma}\label{lem:Q-packing}
Under \eqref{eq:Q-growth}, there is $A\geq1$ such that every
$G\subset B_X(x,R)$, where $R\geq r\geq1$, has a cover by at most
$A(R/r)^Q$ balls of radius $r$ with centers in $G$.
Moreover, $X$ and all its metric subspaces satisfy uniform coarse packing.
\end{lemma}
\begin{proof}
Let $Z\subset B_X(x,R)$ be $r$-separated. The balls $B_X(z,r/3)$,
$z\in Z$, are pairwise disjoint and contained in $B_X(x,R+r/3)$.
If $r\geq3$, their cardinalities are at least $v_-(r/3)^Q$.
If $1\leq r<3$, each contains its center and hence has cardinality
at least $3^{-Q}r^Q$. In both cases \eqref{eq:Q-growth} gives
$\#Z\leq A(R/r)^Q$ with a uniform constant $A$.
A maximal $r$-separated subset of $G$ gives the required internal cover:
any point not in its open $r$-balls could be adjoined to it.
The packing estimate passes to subspaces, since separation and distances
are inherited from $X$.
\end{proof}

The volume assumption does not require every dyadic shell to be nonempty.
The following observation will allow us to handle gaps between occupied
shells without adding a geodesicity assumption.
\begin{lemma}\label{lem:Q-rich-shells}
There are $c_*>0$ and an integer $h\geq1$ such that
\begin{equation*}
 I=\{n\geq1:\ \#S_n\geq c_*2^{Qn}\}
\end{equation*}
meets every sufficiently late interval of $h$ consecutive integers.
\end{lemma}
\begin{proof}
Choose $h$ so that $v_+2^{-hQ}\leq v_-/2$. For $k\geq h$,
\[
 \sum_{j=k-h+1}^{k}\#S_j
 =\#B_X(o,2^k)-\#B_X(o,2^{k-h})
 \geq (v_-/2)2^{Qk}.
\]
Some $j$ in this interval therefore satisfies
$\#S_j\geq(v_-/(2h))2^{Qk}\geq(v_-/(2h))2^{Qj}$.
Take $c_*=v_-/(2h)$.
\end{proof}

We identify all discrete intermediate dimensions of the ambient space with its volume growth exponent and establish a continuity bound for the spectra of its subspaces.

\begin{proof}[Proof of Theorem~\ref{thm:Q-growth}]
Unit balls centered at the points of $E\cap S_n$ give, for $s>Q$,
\[
 \Phi_{\delta,o}^{s,\theta}(E)
 \leq v_+\sum_{n\geq N_\delta}2^{-n(s-Q)}
 \lesssim_s\delta^{s-Q}.
\]
This proves the upper bound $Q$ for all subspaces, uniformly in $\theta$.
For the lower bound for $X$, take $n\in I$ from
Lemma~\ref{lem:Q-rich-shells}. Every cover admissible at $\eta=0$ satisfies
\[
 c_*2^{Qn}\leq\#S_n\leq\sum_i\#B_X(x_i,r_i)
 \leq v_+\sum_i r_i^Q,\qquad 1\leq r_i\leq2^n.
\]
Consequently, for $0\leq s\leq Q$,
\[
 \sum_i(r_i/2^n)^s\geq\sum_i(r_i/2^n)^Q\geq c_*/v_+.
\]
There are infinitely many such shells beyond every $N_\delta$, so
$\Gamma_{\delta,o}^{s,0}(X)=+\infty$. Since the regularized content
includes $\eta=0$, this proves \eqref{eq:Q-full-spectrum}.

For \eqref{eq:Q-modulus}, repeat the proof of
Lemma~\ref{lem:parameter-comparison} using Lemma~\ref{lem:Q-packing}
and replacing $d$ by $Q$. The only geometric input is the bound
$A(r/u)^Q$ for the number of internal radius-$u$ balls needed to cover
a radius-$r$ ball. With $0\leq s<Q$ and $t=Q-(\theta/\phi)(Q-s)$, the identity
$\theta(Q-s)+\phi(t-Q)=0$ gives the same cost comparison.
The supremum over the auxiliary parameter and the passage to critical
exponents are exactly those in Proposition~\ref{thm:continuity}.
If $D_\theta E=Q$, monotonicity and the range bound give the assertion directly.
\end{proof}

\subsection{Sparse clusters and sharp spectra}
For subspaces of $X$, define the lower and upper mass dimensions using
$\#(E\cap B_X(o,2^n))$ in the formulae of Section~\ref{sec:endpoints-intro}.
We now prove the general cluster formula, of which Example~\ref{continuousspectrum} is a special case.

\begin{proof}[Proof of Theorem~\ref{thm:Q-clusters}]
Write $N=N_\delta$ and take $\delta$ small enough that $N\geq n_0$.
For $0<s<Q$, $0\leq\eta\leq\theta$, and $n\in I\cap J_{\delta,\eta}$,
put $u=2^n\delta^\eta$ and $R_n=C_0 2^{\beta n}$.
If $u\geq R_n$, Lemma~\ref{lem:Q-packing} covers $A_n$ by a bounded
number of balls of radius $R_n$ with centers in $A_n$.
If $u<R_n$, it gives at most $A(R_n/u)^Q$ internal balls of radius $u$.
Hence
\[
 \Phi_{n,\delta,o}^{s,\eta}(E_\beta)\lesssim
 \begin{cases}
 2^{-\alpha sn},&u\geq R_n,\\
 2^{-Q\alpha n}\delta^{-\eta(Q-s)},&u<R_n.
 \end{cases}
\]
Empty shells contribute zero. Bounding each of these two sums by the
corresponding full geometric series, uniformly in $\eta$, gives
\begin{equation}\label{eq:Q-cluster-upper}
 \Phi_{\delta,o}^{s,\theta}(E_\beta)
 \lesssim_{s,\beta}\delta^{\alpha s}
             +\delta^{Q\alpha-\theta(Q-s)}.
\end{equation}
For $0\leq\theta\leq\alpha$, both exponents are positive for every
$0<s<Q$. For $\theta>\alpha$, they are positive when
$Q(1-\alpha/\theta)<s<Q$. This proves the upper bound in
\eqref{eq:Q-cluster-spectrum}.

For the lower bound, let $\theta>\alpha$ and
$0\leq s<Q(1-\alpha/\theta)$. Choose $n\in I\cap[N,N+h]$ and define
\begin{equation}\label{eq:Q-gap-parameter}
 \eta_N=\min\{\theta,N/(N+h)\},\qquad u=2^n\delta^{\eta_N}.
\end{equation}
Since $\eta_N\leq N/(N+h)$, we have
$N/\eta_N\geq N+h\geq n$, hence $n\in J_{\delta,\eta_N}$.
Also $0<\eta_N\leq\theta$ and $\eta_N\to\theta$. An arbitrary admissible cover of $A_n$ satisfies
\[
 c_0 2^{Q\beta n}\leq\#A_n
 \leq\sum_i\#B_X(x_i,r_i)\leq v_+\sum_i r_i^Q,
 \qquad1\leq r_i\leq u.
\]
Since $s<Q$,
\[
 \sum_i(r_i/2^n)^s
 \geq2^{-ns}u^{s-Q}\sum_i r_i^Q
 \gtrsim2^{-Q\alpha n}\delta^{-\eta_N(Q-s)}
 \asymp\delta^{Q\alpha-\eta_N(Q-s)}.
\]
The last constants are uniform because $0\leq n-N\leq h$.
Taking infima, using the supremum in the regularized content, and
normalizing logarithms yields
\[
 \liminf_{\delta\downarrow0}
 \frac{\log\Phi_{\delta,o}^{s,\theta}(E_\beta)}{-\log\delta}
 \geq\theta(Q-s)-Q\alpha>0.
\]
This proves the required lower bound. Notice that
\eqref{eq:Q-gap-parameter} also handles $\theta=1$, when the shell
$S_N$ itself might be empty.

For all sufficiently large integers $k$, bounded gaps supply an
$n\in I\cap[k-h,k]$. Therefore
\[
 c2^{Q\beta(k-h)}\leq\#(E_\beta\cap B_X(o,2^k))
 \leq1+C_1\sum_{j=n_0}^{k}2^{Q\beta j}\leq C2^{Q\beta k}.
\]
Taking logarithms proves both mass dimension claims. Finally,
\eqref{eq:Q-sharpness} follows by direct calculation.
\end{proof}

\begin{proof}[Proof of Corollary~\ref{cor:Q-existence}]
Take $I$ from Lemma~\ref{lem:Q-rich-shells} and put $L_n=2^{\beta n}$.
Lemma~\ref{lem:Q-packing} covers $S_n\subset B_X(o,2^n)$ by at most
$A(2^n/L_n)^Q$ balls of radius $L_n$ with centers in $S_n$.
Since $\#S_n\geq c_*2^{Qn}$, one of these balls, centered at $z_n$, obeys
\[
 \#(S_n\cap B_X(z_n,L_n))\geq(c_*/A)L_n^Q.
\]
Set $A_n=S_n\cap B_X(z_n,L_n)$. Its cardinality is at most $v_+L_n^Q$
by \eqref{eq:Q-growth}, so \eqref{eq:Q-cluster-assumptions} holds.
Apply Theorem~\ref{thm:Q-clusters}.
\end{proof}

\begin{remark}\label{rem:Q-groups}
If an infinite group carries a left-invariant metric with
$\#B(e,R)\asymp R^Q$ uniformly for $R\geq1$, then left translation
implies \eqref{eq:Q-growth}. Thus all the preceding results apply to it.
In particular, they apply to any infinite finitely generated group whose
word balls satisfy this two-sided estimate. Theorem~\ref{thm:mainG}
also transfers the dimension formulae to quasi-isometric spaces satisfying
uniform coarse packing.
\end{remark}

\subsection{Independent thinning and a nonconstant random spectrum}
\begin{proof}[Proof of Theorem~\ref{thm:Q-thinning}]
Put $m=Q\beta-\rho>0$. We first establish two uniform estimates on an event of probability one:
for all sufficiently large $n\in I$,
\begin{equation}\label{eq:Q-random-count}
 c2^{mn}\leq\#A_n^\omega\leq C2^{mn},
\end{equation}
and, for every $x\in S_n$ and $1\leq r\leq2^n$,
\begin{equation}\label{eq:Q-random-local}
 \#(A_n^\omega\cap B_X(x,r))\leq C(n+p_nr^Q).
\end{equation}
The constants may be chosen deterministic; the first valid index may
depend on $\omega$.

Indeed, $\#A_n^\omega$ is a sum of independent Bernoulli variables with
mean $\mu_n=p_n\#A_n\asymp2^{mn}$. For a Bernoulli sum $Z$ of mean $\mu$, independence and $1+v\leq e^v$
give $\mathbb E e^{tZ}\leq\exp(\mu(e^t-1))$ for all real $t$.
Exponential Markov inequalities with $t=\log(3/2)$ for the upper tail
and $t=-\log2$ for the lower tail yield
$\mathbb P(|Z-\mu|>\mu/2)\leq2e^{-\mu/12}$.
Applying this to $Z=\#A_n^\omega$ gives
\[
 \mathbb P\{|\#A_n^\omega-\mu_n|>\mu_n/2\}
 \leq2\exp(-c\mu_n).
\]
These probabilities are summable because $m>0$, proving
\eqref{eq:Q-random-count} by the Borel--Cantelli lemma.

To prove \eqref{eq:Q-random-local}, first consider $r=2^j$,
$0\leq j\leq n$, and write
$Y_{n,x,j}=\#(A_n^\omega\cap B_X(x,2^j))$.
Its mean is at most $v_+p_n2^{jQ}$, and independence yields
\[
 \mathbb E\exp(Y_{n,x,j})
 \leq\exp((\mathrm e-1)\mathbb E Y_{n,x,j}).
\]
Consequently, if $K\geq(\mathrm e-1)v_+$, exponential Markov's inequality
implies
\[
 \mathbb P\{Y_{n,x,j}>K(p_n2^{jQ}+n)\}\leq\exp(-Kn).
\]
The growth assumption makes $S_n$ finite. There are at most
$v_+2^{nQ}(n+1)$ pairs $(x,j)$ with $x\in S_n$.
Choose also $K>Q\log2+2$. The union bound makes the probabilities of
any failure in shell $n$ summable. Borel--Cantelli, followed by rounding
an arbitrary $r$ up to the next dyadic radius (at most $2r$), proves
\eqref{eq:Q-random-local}.

Fix a sample for which these estimates hold. The internal covering
estimates used in \eqref{eq:Q-cluster-upper} apply equally to each
$A_n^\omega\subset A_n$, with centers chosen in $A_n^\omega$.
They give the upper bound $F_{\beta,Q}(\theta)$, including the zero
endpoint. Independently, unit-ball covers and
\eqref{eq:Q-random-count} imply, for $s>m$,
\[
 \Phi_{\delta,o}^{s,\theta}(E_{\beta,\rho}^\omega)
 \leq C\sum_{n\geq N_\delta}2^{-n(s-m)}
 \lesssim_s\delta^{s-m}
\]
once $N_\delta$ exceeds the random exceptional indices. This proves
\[
 \uDim_\theta E_{\beta,\rho}^\omega\leq\min\{m,F_{\beta,Q}(\theta)\}.
\]

For the lower bound, let $\alpha=1-\beta$, $\theta>\alpha$, and
$0\leq s<\min\{m,Q(1-\alpha/\theta)\}$. Choose $N,n,\eta_N,u$ as in
\eqref{eq:Q-gap-parameter}, beyond the exceptional indices.
For any admissible cover of $A_n^\omega$, its centers lie in $S_n$ and
$1\leq r_i\leq u\leq2^n$. Hence \eqref{eq:Q-random-local} gives
\[
 \#A_n^\omega\leq C\sum_i(n+p_nr_i^Q)
 \leq C(n+p_nu^{Q-s})\sum_i r_i^s,
\]
where we used $r_i\geq1$ and $s<Q$. By \eqref{eq:Q-random-count},
\begin{align*}
 \sum_i(r_i/2^n)^s
 &\gtrsim\frac{2^{n(m-s)}}{n+p_nu^{Q-s}}\\
 &\gtrsim\min\left\{\frac{2^{n(m-s)}}{n},
             2^{-Q\alpha n}\delta^{-\eta_N(Q-s)}\right\}.
\end{align*}
Since $n=N+O(1)$, $\log n/N\to0$, and $\eta_N\to\theta$, taking infima and normalized
logarithms gives
\[
 \liminf_{\delta\downarrow0}
 \frac{\log\Phi_{\delta,o}^{s,\theta}(E_{\beta,\rho}^\omega)}{-\log\delta}
 \geq\min\{m-s,\theta(Q-s)-Q\alpha\}>0.
\]
This proves \eqref{eq:Q-random-spectrum}.
All these deterministic arguments take place on the single event supplied
by \eqref{eq:Q-random-count}--\eqref{eq:Q-random-local}; thus the conclusion
holds simultaneously for all $\theta$, without an uncountable intersection
of probability-one events. The bounded gaps in $I$, together with
\eqref{eq:Q-random-count}, give both mass dimensions just as in the proof
of Theorem~\ref{thm:Q-clusters}. Solving
$Q(1-\alpha/\theta)=Q\beta-\rho$ gives
\eqref{eq:Q-random-transition} and the stated branches.
\end{proof}

\begin{remark}
When $\rho=0$, every point is retained and
Theorem~\ref{thm:Q-thinning} recovers Theorem~\ref{thm:Q-clusters}.
For $0<\rho<Q\beta$, independent thinning leaves the onset
$\theta=1-\beta$ and the initial increasing branch unchanged, but introduces
a plateau before $\theta=1$. Thus the endpoint records the number of
retained points, whereas the increasing branch also records their
concentration within the original clusters.
\end{remark}

\subsection{An application to the Heisenberg group}
We use the word metric on the vertex set of the Cayley graph, rather than
the Euclidean coordinate metric on $\Z^3$.
Let $H=\mathbb H(\Z)$ be the discrete Heisenberg group. In coordinates,
\[
H=\Z^3,\qquad
(x,y,z)(x',y',z')=(x+x',y+y',z+z'+xy').
\]
Write $e=(0,0,0)$, $a=(1,0,0)$, $b=(0,1,0)$, and
$c=[a,b]=aba^{-1}b^{-1}=(0,0,1)$. Let $d_H$ be the word metric for
$\{a,a^{-1},b,b^{-1}\}$, and put $|g|=d_H(e,g)$. In particular,
\begin{equation}\label{eq:heisenberg-commutator}
[a^p,b^q]=c^{pq}\qquad(p,q\in\Z).
\end{equation}

\begin{lemma}\label{lem:heisenberg-growth}
There are constants $v_-,v_+,A>0$ such that
\begin{equation*}
v_-R^4\leq\#B_H(g,R)\leq v_+R^4
\qquad(g\in H,\ R\geq1).
\end{equation*}
Moreover, for $R\geq r\geq1$, every subset $G\subset B_H(g,R)$ can be
covered by at most $A(R/r)^4$ balls of radius $r$ with centers in $G$.
In particular, $H$ and all its metric subspaces satisfy uniform coarse
packing.
\end{lemma}

\begin{proof}
Left translations are isometries, so it suffices first to count balls
centered at $e$. Along a word of length at most $R$, the coordinates obey
$|x|,|y|\leq R$ and $|z|\leq R^2$: right multiplication by $a^{\pm1}$
changes $x$ by $\pm1$, while right multiplication by $b^{\pm1}$ changes
$y$ by $\pm1$ and $z$ by $\pm x$. This proves the upper bound.

For the lower bound, fix an integer $m\geq1$. If $0\leq k\leq m^2$,
write $k=mq+r$, where $0\leq q\leq m$ and $0\leq r<m$. By
\eqref{eq:heisenberg-commutator},
\[
c^k=[a^m,b^q][a^r,b],\qquad |c^k|\leq6m+2.
\]
For negative $k$, take the inverse of the word representing $c^{-k}$;
this preserves its length. Since
$(x,y,z)=b^ya^xc^z$, every integer point with
\[
|x|,|y|\leq m,\qquad |z|\leq m^2
\]
has word length at most $8m+2$. Taking $m=\lfloor R/16\rfloor$ for
$R\geq32$ gives at least a constant times $R^4$ points in the open ball
$B_H(e,R)$. Decreasing $v_-$ handles $1\leq R<32$.

The internal covering and coarse packing assertions now follow from
Lemma~\ref{lem:Q-packing}, applied with $Q=4$.
\end{proof}

\begin{remark}\label{rem:heisenberg-non-euclidean}
The Cayley graph $(H,d_H)$ is not quasi-isometric to $\R^m$ for any
$m$. Indeed, Pansu's asymptotic-cone theorem \cite{Pansu83} identifies its
asymptotic cone with the real Heisenberg group equipped with a
Carnot--Carath\'eodory Finsler metric; see also
\cite[Section~1.2, Theorem 2]{DM14}. The cone is homeomorphic to
$\R^3$, so has topological dimension $3$. Its dilations
$(x,y,z)\mapsto(rx,ry,r^2z)$ multiply distances by $r$ and Haar volume
by $r^4$. Left invariance then gives Ahlfors $4$-regularity, and hence
Hausdorff dimension $4$. On rescaling a quasi-isometry by factors tending to infinity, its
additive error and coarse-surjectivity error tend to zero. It therefore
induces a bi-Lipschitz equivalence of asymptotic cones. For $\R^m$,
this cone is $\R^m$ itself, forcing both $m=3$ and
$m=4$. Thus the example is non-Euclidean even at the level of
quasi-isometry. The elementary estimates in Lemma~\ref{lem:heisenberg-growth}
are sufficient for all dimension computations below.
\end{remark}

\begin{corollary}
\label{thm:heisenberg-spectrum}
For every $\theta\in[0,1]$,
\[
 \lDim_\theta H=\uDim_\theta H=4.
\]
Every unbounded $E\subset H$ satisfies
\begin{equation}\label{eq:heisenberg-modulus}
 0\leq D_\phi E-D_\theta E
 \leq(1-\theta/\phi)(4-D_\theta E),\qquad0<\theta<\phi\leq1.
\end{equation}
For $0<\beta<1$, choose $n_0$ sufficiently large and put
\begin{equation*}
 g_n=a^{3\cdot2^{n-2}},\quad L_n=\lfloor2^{\beta n}\rfloor,\quad
 E_n=B_H(g_n,L_n),\qquad
 E_\beta^H=\{e\}\cup\bigcup_{n\geq n_0}E_n.
\end{equation*}
Then
\begin{equation*}
 \lDim_\theta E_\beta^H=\uDim_\theta E_\beta^H
 =F_\beta(\theta):=F_{\beta,4}(\theta),
\end{equation*}
both mass dimensions equal $4\beta$, and for $1-\beta\leq\theta<\phi\leq1$,
\begin{equation*}
 F_\beta(\phi)-F_\beta(\theta)
 =(1-\theta/\phi)(4-F_\beta(\theta)).
\end{equation*}
If each point of $E_n$ is independently retained with probability
$2^{-\rho n}$, where $0\leq\rho<4\beta$, the resulting set
$E_{\beta,\rho}^{H,\omega}$ has, almost surely and simultaneously for all
$\theta$, lower and upper spectra $G_{\beta,\rho,4}(\theta)$ and both
mass dimensions equal to $4\beta-\rho$.
\end{corollary}
\begin{proof}
Lemma~\ref{lem:heisenberg-growth} verifies \eqref{eq:Q-growth} with $Q=4$.
The full-group assertion and \eqref{eq:heisenberg-modulus} therefore follow
from Theorem~\ref{thm:Q-growth}.
The homomorphism $(x,y,z)\mapsto x$ shows that $|a^k|\geq|k|$, while
the word $a^k$ gives the reverse inequality. Thus $|g_n|=3\cdot2^{n-2}$.
Take $n_0\geq3$ so large that $2\leq L_n\leq2^{n-4}$ for $n\geq n_0$.
The triangle inequality then gives $E_n\subset S_n^H(e)$, and
\begin{equation*}
 \#E_n\asymp L_n^4\asymp2^{4\beta n}.
\end{equation*}
Hence these complete balls satisfy \eqref{eq:Q-cluster-assumptions},
with $I$ containing every sufficiently large integer.
Theorems~\ref{thm:Q-clusters} and~\ref{thm:Q-thinning} give all the remaining
claims.
\end{proof}

For example, when $\beta=1/2$ and $\rho=1$, the random Heisenberg subset
has, almost surely, the spectrum
\[
 G_{1/2,1,4}(\theta)=
 \begin{cases}
 0,&0\leq\theta\leq1/2,\\
 4-2/\theta,&1/2<\theta<2/3,\\
 1,&2/3\leq\theta\leq1.
 \end{cases}
\]
The deterministic spectra are illustrated in
Figure~\ref{fig:heisenberg-spectra}.

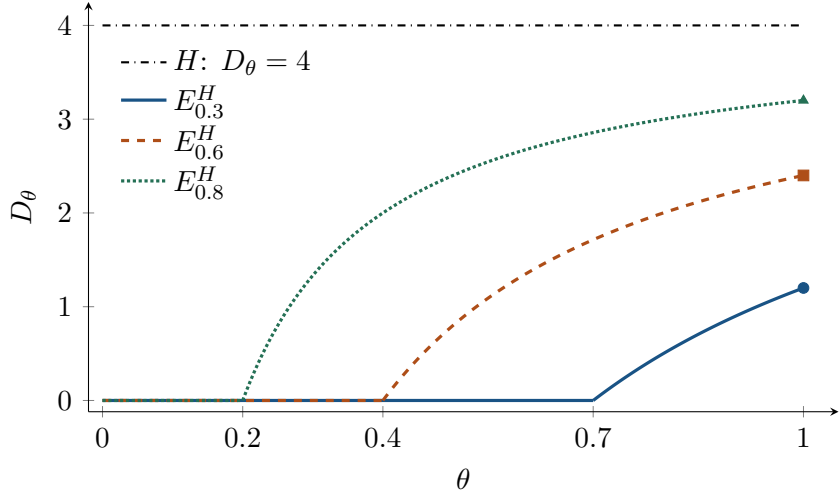
\begin{figure}[htbp]
\centering
\begin{tikzpicture}
\begin{axis}[
 width=11.5cm,height=7cm,axis lines=left,
 xmin=-0.02,xmax=1.05,ymin=-0.12,ymax=4.25,
 xlabel={$\theta$},ylabel={$D_\theta$},
 xtick={0,0.2,0.4,0.7,1},ytick={0,1,2,3,4},
 tick label style={font=\small},label style={font=\small},
 legend style={at={(0.03,0.91)},anchor=north west,font=\small,
 draw=none,fill=white},legend cell align=left]
\addplot[black,dashdotted,thick] coordinates {(0,4) (1,4)};
\addlegendentry{$H$: $D_\theta=4$}
\addplot[spectrumblue,very thick] coordinates {(0,0) (0.7,0)};
\addlegendentry{$E_{0.3}^H$}
\addplot[spectrumblue,very thick,domain=0.7:1,samples=101,forget plot]
 {4*(1-0.7/x)};
\addplot[spectrumorange,very thick,dashed] coordinates {(0,0) (0.4,0)};
\addlegendentry{$E_{0.6}^H$}
\addplot[spectrumorange,very thick,dashed,domain=0.4:1,samples=121,forget plot]
 {4*(1-0.4/x)};
\addplot[spectrumgreen,very thick,densely dotted] coordinates {(0,0) (0.2,0)};
\addlegendentry{$E_{0.8}^H$}
\addplot[spectrumgreen,very thick,densely dotted,domain=0.2:1,samples=151,forget plot]
 {4*(1-0.2/x)};
\addplot[only marks,mark=*,spectrumblue,forget plot] coordinates {(1,1.2)};
\addplot[only marks,mark=square*,spectrumorange,forget plot] coordinates {(1,2.4)};
\addplot[only marks,mark=triangle*,spectrumgreen,forget plot] coordinates {(1,3.2)};
\end{axis}
\end{tikzpicture}
\caption{The constant spectrum of the Heisenberg group and the spectra
$F_\beta$ of its sparse ball clusters. The transition is at
$\theta=1-\beta$, and the endpoint value is $4\beta$. Both the lower and
upper dimensions equal the displayed functions.}
\label{fig:heisenberg-spectra}
\end{figure}

\begin{corollary}\label{cor:heisenberg-qi-application}
Let $Y$ be an unbounded metric space satisfying uniform coarse packing.
If $Y$ is quasi-isometric to $H$, then both its discrete intermediate
dimensions equal $4$ for every $\theta$. If $Y$ is quasi-isometric to
$E_\beta^H$, then both dimensions equal $F_\beta(\theta)$.
In particular, these conclusions are unchanged when the ambient word
metric on $H$ is replaced by that associated with any other finite
symmetric generating set, or when $E_\beta^H$ is replaced by a subset of
$H$ at finite Hausdorff distance from it.
\end{corollary}

\begin{proof}
Apply Theorem~\ref{thm:mainG}, using
Lemma~\ref{lem:heisenberg-growth} for the source spaces.
Two finite generating sets give bi-Lipschitz equivalent word metrics;
the growth and packing estimates therefore hold for each choice.
Two subsets at finite Hausdorff distance are quasi-isometric by choosing
nearby points, and subspaces of $H$ satisfy uniform coarse packing.
\end{proof}

\subsection{An application to random-walk ranges}
The preceding thinning theorem produces nonconstant random spectra.
A complementary application to a standard stochastic process follows
from the endpoint identities and monotonicity. The classical discrete
Hausdorff dimension theorem for random walks is due to Barlow and Taylor
\cite{BT}. We use the independently accessible formulation in
\cite[Theorem 1.1]{XZ}, specialized to constant conductances.

\begin{corollary}\label{cor:simple-random-walk}
Let $(Z_k)_{k\geq0}$ be simple symmetric nearest-neighbor random walk on
$\Z^d$, started at the origin, and let
$\mathcal R=\{Z_k:k\geq0\}$. With the metric induced from $\R^d$, almost
surely and simultaneously for all $\theta\in[0,1]$,
\begin{equation}\label{eq:random-walk-spectrum}
 \lDim_\theta\mathcal R=\uDim_\theta\mathcal R=\min\{d,2\}.
\end{equation}
\end{corollary}
\begin{proof}
For $d=1,2$, recurrence (\cite[Theorem 4.1.1]{LL10}) and
irreducibility imply that every lattice point
is visited almost surely. Countability gives $\mathcal R=\Z^d$ on a single
event of probability one. Theorem~\ref{thm:Q-growth}, with $Q=d$, and
Lemma~\ref{lem:euclidean-intrinsic} prove the assertion in this case.

Suppose $d\geq3$. Write
$g(0,x)=\sum_{k\geq0}\mathbb P_0\{Z_k=x\}$ for the Green function.
By \cite[Theorem 4.3.1]{LL10},
$g(0,x)\asymp(1+|x|)^{2-d}$. Taking every conductance equal to one
in \cite[Theorem 1.1]{XZ} gives $\hdim\mathcal R=2$ almost surely:
the embedded jump chain is simple symmetric random walk, and the
non-explosive constant-rate holding times leave its range unchanged.
We record the upper mass bound explicitly. Let
$T_x=\inf\{k\geq0:Z_k=x\}$ and
$M_R=\#(\mathcal R\cap B(0,R))$.
The strong Markov property and translation invariance imply
$\mathbb P_0\{T_x<\infty\}=g(0,x)/g(0,0)$. Hence
\[
 \mathbb E_0 M_R
 =\sum_{x\in\Z^d\cap B(0,R)}\mathbb P_0\{T_x<\infty\}
 \lesssim\sum_{|x|<R}(1+|x|)^{2-d}\lesssim R^2.
\]
The final bound follows by summing over dyadic lattice annuli;
the interchange of expectation and the nonnegative point indicators
is justified by Tonelli's theorem.
For every $\varepsilon>0$, Markov's inequality yields
\[
 \mathbb P_0\{M_{2^n}>2^{(2+\varepsilon)n}\}
 \lesssim2^{-\varepsilon n}.
\]
Apply Borel--Cantelli for $\varepsilon=1/j$, $j\in\N$, and intersect
the resulting countably many events to obtain
$\dim_{\mathrm{UM}}\mathcal R\leq2$ almost surely.
On the event where this bound and $\hdim\mathcal R=2$ both hold,
Propositions~\ref{gjH} and~\ref{dBD}, and parameter monotonicity, give
\[
 2=\hdim\mathcal R
 \leq\lDim_\theta\mathcal R
 \leq\uDim_\theta\mathcal R
 \leq\uDim_1\mathcal R
 =\dim_{\mathrm{UM}}\mathcal R\leq2
\]
for every $\theta\in[0,1]$. This proves simultaneity as well as
\eqref{eq:random-walk-spectrum}.
\end{proof}

\begin{remark}
Corollary~\ref{cor:simple-random-walk} is an application of known
random-walk endpoint results to the present family of dimensions.
In contrast, Theorem~\ref{thm:Q-thinning} computes an entire nonconstant
spectrum from uniform probabilistic estimates on the local occupation
counts. These two examples illustrate respectively the collapse of the
spectrum when the endpoint dimensions agree and its variation when spatial
clustering and point counts impose different covering constraints.
\end{remark}

\section*{Statements and Declarations}
\subsection*{Funding}
This work was supported by the National Natural Science Foundation of China
(Grant Nos. 12171107 and 12571096).

\subsection*{Data Availability}
This article is theoretical. No research datasets were generated or analysed.

\subsection*{Competing Interests}
The authors declare that they have no financial or non-financial interests
relevant to the content of this article.

\subsection*{Use of generative AI}
During preparation of this manuscript, the authors used ChatGPT (OpenAI) to assist with language
revision. Responsibility for the mathematical
arguments, citations, and final text rests with the authors.

\end{document}